\documentclass[12pt]{amsart}
\usepackage{amsmath,amsthm,bm}
\usepackage[psamsfonts]{amssymb}
\usepackage{type1cm}
\usepackage[dvipdfmx]{graphicx,xcolor}

\theoremstyle{plain}
\newtheorem{theorem}{Theorem}
\newtheorem{corollary}{Corollary}[theorem]
\newtheorem{lemma}{Lemma}
\newtheorem{proposition}{Proposition}
\theoremstyle{remark}
\newtheorem{remark}{Remark}

\newtheorem{remarkone}{Remark 1}

\newtheorem{remarktwo}{Remark 2}

\def\al{\alpha}
\def\be{\beta}

\def\vGa{\varGamma}
\def\ga{\gamma}

\def\de{\delta}

\def\vep{\varepsilon}
\def\ze{\zeta}
\def\et{\eta}

\def\ka{\kappa}
\def\la{\lambda}

\def\rh{\rho}

\def\vSi{\varSigma}
\def\si{\sigma}

\def\ta{\tau}

\def\ph{\phi}

\def\ps{\psi}

\def\om{\omega}
\def\fH{\mathfrak{H}}

\bmdefine{\bon}{1}
\bmdefine{\ba}{a}
\bmdefine{\bb}{b}
\bmdefine{\bc}{c}
\bmdefine{\bd}{d}
\bmdefine{\bde}{e}
\bmdefine{\bm}{m}
\bmdefine{\bs}{s}
\bmdefine{\bx}{x}
\bmdefine{\by}{y}
\bmdefine{\bz}{z}
\bmdefine{\bal}{\alpha}
\bmdefine{\bbe}{\beta}
\bmdefine{\bga}{\gamma}
\bmdefine{\bom}{\omega}
\bmdefine{\bzr}{0}
\def\fr{\frac}

\def\wt#1{\widetilde{#1}}
\def\wh#1{\widehat{#1}}

\def\re{\operatorname{Re}}
\def\im{\operatorname{Im}}
\def\cA{\mathcal{A}}
\def\cB{\mathcal{B}}

\def\cC{\mathcal C}

\def\cS{\mathcal{S}}

\def\bbZ{\mathbb{Z}}

\def\BZ2{\mathbb{Z}^2}
\def\bZ2{\mathbb{Z}^2}
\def\BR{\mathbb{R}}
\def\bR{\mathbb{R}}
\def\BC{\mathbb{C}}
\def\bC{\mathbb{C}}

\def\dps{\displaystyle}

\def\del{\partial}

\def\sump{\mathop{\sideset{}{'}\sum}}

\def\sgn{\operatorname{sgn}}

\def\bsm{\begin{smallmatrix}}
\def\esm{\end{smallmatrix}}

\def\lg{\langle}
\def\rg{\rangle}
\def\lf{\lfloor}
\def\rf{\rfloor}
\def\tms{\times}
\def\res{\mathrm{Res}}

\def\wt{\widetilde}
\def\bg{\begin}
\def\ed{\end}

\def\vG#1#2{\vGa\Bigl(\begin{matrix}#1\\ #2\end{matrix}\Bigr)}
\def\1F1#1#2#3{{}_1F_1\Bigl(\begin{matrix}#1\\ #2\end{matrix};#3\Bigr)}
\def\t1F1#1#2#3{{}_1F_1(\begin{smallmatrix}#1\\ #2\end{smallmatrix};#3)}
\def\2F1#1#2#3{{}_2F_1\Bigl(\begin{matrix}#1\\ #2\end{matrix};#3\Bigr)}

\def\R#1#2#3#4{R_{#1}\Bigl(\begin{matrix}#2\\ #3\end{matrix};#4\Bigr)}
\def\R*#1#2#3#4{R^{\ast}_{#1}\Bigl(\begin{matrix}#2\\ #3\end{matrix};#4\Bigr)}

\def\Fz2#1#2#3#4{F_{\mathbb{Z}^2}(#1;#2;#3;#4)}
\def\tFz2#1#2#3#4{\textstyle F_{\mathbb{Z}^2}(#1;#2;#3;#4)}

\renewcommand{\subjclass}[1]{\footnote[0]{2020 \textit{Mathematics Subject Classification.} #1.}}
\begin{document}
\title[A class of generalized holomorphic Eisenstein series]{Asymptotic expansions for a 
class of generalized holomorphic Eisenstein series, Ramanujan's formula for $\ze(2k+1)$, 
Weierstra{\ss}' elliptic and allied functions} 
\author[Katsurada and Noda]{Masanori KATSURADA and Takumi Noda}
\address{\footnotesize Department of Mathematics, Hiyoshi Campus, Keio 
University, 4--1--1 Hiyoshi, Kouhoku-ku, Yokohama 223--8521, Japan;
Department of Mathematics, College of Engineering, Nihon University, 1 Nakagawara, 
Tokusada, Tamuramachi, K{\^o}riyama, Fukushima 936--8642, Japan}
\email{\tt katsurad@z3.keio.jp; noda.takumi@nihon-u.ac.jp}
\begin{abstract}
For a class of generalized holomorphic Eisenstein series, we establish complete 
asymptotic expansions (Theorems~1~and~2), which together with the explicit 
expression of the latter remainder (Theorem~3), naturally transfer to several new 
variants of the celebrated formulae of Euler and of Ramanujan for specific values 
of the Riemann zeta-function (Theorem~4 and Corollaries~4.1--4.5), and to various 
modular type relations for the classical Eisenstein series of any even integer weight 
(Corollary~4.6) as well as for Weierstra{\ss}' elliptic and allied functions (Corollaries~4.7--4.9). 
Crucial r{\^o}les in the proofs are played by certain Mellin-Barnes type integrals, 
which are manipulated with several properties of confluent hypergeometric functions.
\end{abstract} 
\maketitle
\section{Introduction}
Let $\mathfrak{H}^{\pm}$ be the complex upper and lower 
half-planes respectively, where the argument of each sheet is chosen as 
\begin{align*}
\mathfrak{H}^+=\{z\in\wt{\mathbb{C}^{\times}}\mid 0<\arg z<\pi\}
\qquad
\text{and}
\qquad
\mathfrak{H}^-=\{z\in\wt{\mathbb{C}^{\times}}\mid -\pi<\arg z<0\}.
\end{align*} 
Throughout the paper, $s$ is a complex variable, $\al$, $\be$, $\mu$ and $\nu$ real parameters, 
$z$ a complex parameter with $z\in\mathfrak{H}^{\pm}$. It is frequently used in the sequel the  notations $e(s)=e^{2\pi is}$ and $\vep(z)=\sgn(\arg z)$ for $|\arg z|>0$, and the parameter 
$\ta=e^{\mp\vep(z)\pi i/2}z$, where $\ta$ varies within the sector $|\arg\ta|<\pi/2$. 

We now define for $z\in\fH^+$ the generalized Eisenstein series $F^{\pm}_{\mathbb{Z}^2}(s;\al,\be;\mu,\nu;z)$ by 
\begin{align*}
F_{\mathbb{Z}^2}^{\pm}(s;\al,\be;\mu,\nu;z)
=\sump_{m,n=-\infty}^{\infty}\fr{e\{(\al+m)\mu+(\be+n)\nu\}}{\{\al+m+(\be+n)z\}^{s}}
\qquad(\re s=\si>2),
\tag{1.1}
\end{align*}
where the primed summation symbols hereafter indicate omission of the impossible terms of the form $1/0^s$, and the branch of each summand is chosen such that $\arg\{(\al+m)+(\be+n)z\}$ falls within the range $]-\pi,\pi]$ in $F^+_{\mathbb{Z}^2}$, while within $[-\pi,\pi[$ in $F^-_{\mathbb{Z}^2}$. The main object of study is the arithmetical mean 
\begin{align*}
\Fz2{s}{\al,\be}{\mu,\nu}{z}
=\fr12\bigl\{F^+_{\mathbb{Z}^2}(s;\al,\be;\mu,\nu;z)
+F^-_{\mathbb{Z}^2}(s;\al,\be;\mu,\nu;z)\bigr\},
\tag{1.2}
\end{align*}
for which we shall show that complete asymptotic expansions exist when both $\ta\to\infty$ (Theorem~1) and $\ta\to0$ (Theorem~2), whose proofs, by means of Mellin-Barnes type integrals, 
lead us to extract exponentially small order terms from the latter remainder (Theorem~3). 

Let $\ze(s)$ denote the Riemann zeta-function. The asymptotic series, or even its first 
derivative, in Theorem~2 in fact terminates up to finite terms if $s$ is at any integer point; 
this, combined with Theorems~1~and~3, naturally transfers to several new variants of the 
celebrated formulae of Euler and of Ramanujan for specific values of $\ze(s)$ (Theorem~4 and  Corollaries~4.1--4.5), and also to (quasi) modular relations for the classical Eisenstein series of integer weights (Corollary~4.6) as well as for Weierstra{\ss}' elliptic and allied functions (Corollaries~4.7--4.9).  It is worth noting that a hidden (but crucial) r{\^o}le is played by the connection formula (2.25) below for Kummer's confluent hypergeometric functions in producing various Ramanujan type formulae for specific values of zeta-functions, for the classical Eisenstein series as well as for Weierstra{\ss}' elliptic and allied functions.

We give here an overview of the research related to holomorphic and non-holomorphic Eisenstein 
series of complex variable(s). Lewittes \cite{lewittes1971} first obtained a transformation 
formula for $F(s;z)=\Fz2{s}{0,0}{0,0}{z}$, which was applied to show a modular relation connecting 
$F(2;z)$ and $F(2,-1/z)$; this transformation formula can be viewed as a prototype of our 
Theorem~1. He established further in \cite{lewittes1972} a transformation formula for a more general 
$\Fz2{s}{\al,\be}{0,0}{z}$, which was extensively applied to study its modular relations when the 
modular group $\mathit{SL}_2(\bbZ)$ acts on the associated parameter $z\in\fH^+$. Subsequent 
research was made by Berndt~\cite{berndt1973}, who especially treated in this respect a class of 
generalized Dedekind eta function and Dedekind sums. Berndt \cite{berndt1977} made further 
research into this direction in connection with Euler's and Ramanujan's formulae for specific values 
of $\ze(s)$, while he \cite{berndt1975} also studied certain character analogues of 
$\Fz2{s}{\al,\be}{0,0}{z}$ to show Ramanujan type formulae for Dirichlet $L$-functions. 
Matsumoto~\cite{matsumoto2003}, on the other hand, more recently derived complete asymptotic 
expansions for $F(s;z)$ when both $z\to0$ and $z\to\infty$ through $\fH^+$; the former can be 
viewed as a prototype of our Theorem~2. A transformation formula for a two variable analogue 
of (1.1) was obtained by Lim~\cite{lim2014}, while the first author \cite{katsurada2015} derived complete asymptotic expansions for a generalized {\it non}-holomorphic Eisenstein series of the form
\begin{align*}
\ps_{\bbZ^2}(s;\al,\be;\mu,\nu;z)
&=\sump_{m,n=-\infty}^{\infty}\fr{e\{(\al+m)\mu+(\be+n)\nu\}}{|\al+m+(\be+n)z|^{2s}}
\qquad(\si>1)
\end{align*}   
both as $z\to0$ and $z\to\infty$ through $\fH^+$. It has fairly recently been shown by the authors  \cite{katsurada-noda2017} that complete asymptotic expansions exist for a two variable analogue of $F(s;z)$, when the associated parameters $\bz=(z_1,z_2)$ vary within the polysector $(\fH^{\pm})^2$, so as that the distance $|z_2-z_1|$ becomes both small and large.

The paper is organized as follows. Our main theorems (Theorems~1--4) are stated in the next section, while Section~3 is devoted to presenting the results on the classical Eisenstein series, 
Weierstra{\ss}' elliptic and allied functions. The proofs of Theorems~1, 2, 3 and 4 are given in 
Sections 4, 5, 6 and 7 respectively, while a major portion of the corollaries to Theorems are shown in Section~8.  
\section{Main results}

Prior to the statement of our main formulae, several necessary notations and results are prepared. 

Let $\vGa(s)$ be the gamma function, $(s)_n=\vGa(s+n)/\vGa(s)$ for $n\in\bbZ$ the shifted factorial of $s$, and write 
\begin{align*}
\vG{\al_1,\ldots,\al_m}{\be_1,\ldots,\be_n}
=\fr{\prod_{h=1}^m\vGa(\al_h)}{\prod_{k=1}^n\vGa(\be_k)}
\end{align*}
for $\al_h,\be_k\in\BC$ $(h=1,\ldots,m$; $k=1,\ldots,n)$. We here introduce for a variable $r\in\BC$ and parameters $\ga,\ka\in\BR$ the Lerch zeta-function $\ph(r,\ga,\ka)$, together with its companion 
$\ps(r,\ga,\ka)$, defined by 
\begin{align*}
\ph(r,\ga,\ka)
&=\sum_{-\ga<k\in\bbZ}\fr{e(k\ka)}{(\ga+k)^r}
\qquad(\re r>1),
\tag{2.1}\\
\ps(r,\ga,\ka)
&=\sum_{-\ga<k\in\bbZ}\fr{e\{(\ga+k)\ka\}}{(\ga+k)^r}
=e(\ga\ka)\ph(r,\ga,\ka),
\tag{2.2}
\end{align*}
which can be continued to entire functions if $\ka\notin\bbZ$, while for $\ka\in\bbZ$ the former (or for $\ka=0$ the latter) reduces to the Hurwitz zeta-function 
\begin{align*}
\ze(r,\ga)
&=\sum_{-\ga<k\in\bbZ}\fr{1}{(\ga+k)^r},
\qquad
(\re r>1),
\tag{2.3}
\end{align*} 
also for $\ga=0$ the former (or for $\ga\in\bbZ$ the latter) to the exponential zeta-function 
\begin{align*}
\ze_{\ka}(r)=\sum_{k=1}^{\infty}\fr{e(k\ka)}{k^r}
\qquad(\re r>1),
\tag{2.4}
\end{align*}
and hence further to $\ze(r)=\ze(r,\ga)=\ze_{\ka}(r)$ if $\ga,\ka\in\bbZ$. We remark here that the definition in (2.1) or (2.2) differs slightly from the original, which asserts, for $\ga,\ka\in\bR$ with $\ga\geq0$, 
\begin{align*}
\ph_0(r,\ga,\ka)
&=\sump_{k=0}^{\infty}\fr{e(k\ka)}{(\ga+k)^r}
\qquad(\re r>1),
\tag{2.5}\\
\ps_0(r,\ga,\ka)
&=\sump_{k=0}^{\infty}\fr{e\{(\ga+k)\ka\}}{(\ga+k)^r}
=e(\ga\ka)\ph_0(r,\ga,\ka). 
\tag{2.6}
\end{align*}  

Let $\lg x\rg=x-\lf x\rf $ denote hereafter the fractional part of $x\in\bR$, and use the convention
\begin{align*}
\lg x\rg'=1-\lg-x\rg=
\begin{cases}
\lg x\rg&\quad\text{if $x\notin\bbZ$},\\
1&\quad\text{if $x\in\bbZ$}.
\end{cases}
\tag{2.7}
\end{align*}
One can in fact show that the classical functional equation for (2.5)  (cf.~\cite{lerch1887}\cite{lipschitz1889}\cite[(6.15)]{katsurada2020}) is transferred to that for (2.1) or (2.2) in the following form, which allows to 
relax the original restrictions on the range of parameters such as $\ga,\ka\in[0,1]$.  
\begin{proposition}
For any $\ga,\ka\in\bR$ we have, in the whole $r$-plane $\bC$,
\begin{align*}
\ps(r,\ga,\ka)
&=\fr{\vGa(1-r)}{(2\pi)^{1-r}}
\bigl\{e^{\pi i(1-r)/2}\ph(1-r,\ka,-\ga)
+e^{-\pi i(1-r)/2}\ph(1-r,-\ka,\ga)\bigr\}
\tag{2.8}\\
&=e(\ga\ka)\fr{\vGa(1-r)}{(2\pi)^{1-r}}
\bigl\{e^{\pi i(1-r)/2}\ps(1-r,\ka,-\ga)
+e^{-\pi i(1-r)/2}\ps(1-r,-\ka,\ga)\bigr\}.
\end{align*}
\end{proposition}
\begin{proof}
From the relations, following from (2.1), (2.2) and (2.6) (with the convention of primed summation symbols), for any $\ga,\ka\in\BR$ and $l\in\bbZ$, 
\begin{align*}
&\ph(r,\ga,\ka)
=\ph(r,\ga,\ka+l)
\tag{2.9}\\
&\ps(r,\ga,\ka)
=\ps(r,\lg\ga\rg,\ka)=\ps_0(r,\lg\ga\rg,\ka)
=\ps_0(r,\lg\ga\rg',\ka),\tag{2.10}
\end{align*}
we see $\ps(r,\ga,\ka)=e(\lg\ga\rg\ka)\ph_0(r,\lg\ga\rg,\lg\ka\rg)$, and hence find from the classical functional equation for (2.5) with (2.6) that 
\begin{align*}
\ps(r,\ga,\ka)
&=e(\lg\ga\rg\ka)\fr{\vGa(1-r)}{(2\pi)^{1-r}}
\bigl\{e^{\pi i(1-r)/2}\ps_0(1-r,\lg\ka\rg,-\lg\ga\rg)
\tag{2.11}\\
&\quad+e^{-\pi i(1-r)/2}\ps_0(1-r,1-\lg\ka\rg,\lg\ga\rg\bigr\}.
\end{align*}
The right side of (2.11) becomes, from (2.10) and $1-\lg\ka\rg=\lg-\ka\rg'$ for $\ka\in\bR$, by (2.7),
\begin{align*}
\fr{\vGa(1-r)}{(2\pi)^{1-r}}\bigl\{e^{\pi i(1-r)/2}\ph(1-r,\ka,-\lg\ga\rg)
+e^{-\pi i(1-r)/2}\ph(1-r,-\ka,\lg\ga\rg)\bigr\},
\end{align*} 
which further equals, again by (2.6) and (2.9), the right sides of (2.8).  
\end{proof}

We next introduce, for a (base) complex parameter $q$ with $|q|<1$, and for real parameters $\ga$, 
$\de$, $\ka$ and $\la$, a $q$-series of the form
\begin{align*}
\cS_r(\ga,\de;\ka,\la;q)
&=\sum_{\bsm-\ga<k\in\bbZ\\ -\de<l\in\bbZ\esm}
\fr{e\{(\ga+k)\ka+(\de+l)\la\}}{(\de+l)^r}q^{(\ga+k)(\de+l)}\tag{2.12}\\
&=e(\lg\ga\rg'\ka)\sum_{-\de<l\in\bbZ}
\fr{e\{(\de+l)\la\}q^{\lg\ga\rg'(\de+l)}}{(\de+l)^r\{1-e(\ka)q^{\de+l}\}}.
\end{align*}

We proceed to state our first main result, which gives for (1.2) a transformation formula or an asymptotic expansion in the descending order of $\ta$ as $\ta\to\infty$. 
\bg{theorem}
Let $\al$, $\be$, $\mu$ and $\nu$ be real parameters, $q=e(z)=e^{-2\pi \ta}$ for any $z=i\ta\in\fH^+$ with $|\arg\ta|<\pi/2$, and set 
\begin{align*}
\cA(s,\al,\mu)
&=\cos(\pi s)\ps(s,-\al,-\mu)+\ps(s,\al,\mu)\tag{2.13}\\
&=e(\al\mu)\fr{(2\pi)^s}{2\vGa(s)}
\bigl\{e^{-\pi is/2}\ps(1-s,-\mu,\al)+e^{\pi is/2}\ps(1-s,\mu,-\al)\bigr\},
\end{align*}
which is holomorphic for all $s\in\BC$. Here the second equality is valid by {\rm(}2.8{\rm)}. 
Then we have  
\begin{align*}
\Fz2{s}{\al,\be}{\mu,\nu}{z}
&=\de(\be)\cA(s,\al,\mu)+e(\al\mu)\fr{(2\pi)^s}{\vGa(s)}
\bigl\{e^{-\pi is/2}\cS_{1-s}(\be,-\mu;\nu,\al;q)\}\tag{2.14}\\
&\quad+e^{\pi is/2}\cS_{1-s}(-\be,\mu;-\nu,-\al;q)\bigr\},
\end{align*}
where the $q$-series on the right side converge absolutely for all $s\in\bC${\rm;} this provides the holomorphic continuation of the left side to the whole $s$-plane $\bC$. 
\ed{theorem}
\begin{remarkone}
The case $(\mu,\nu)=(0,0)$ of (2.14) was first established by 
Lewittes~\cite[Theorem~1]{lewittes1972} in terms of contour integrations.
\end{remarkone}
\begin{remarktwo}
The $q$-series on the right side of (2.14) give the (convergent) asymptotic series in the descending order of $\ta$ as $\ta\to\infty$ through $|\arg\ta|<\pi/2$, since each term of the $q$-series is of order $\asymp \exp\{-2\pi\lg\pm\be\rg'(\mp\mu+m)\ta\}$ (for $m>\pm\mu$) as $\ta\to\infty$.
\end{remarktwo}

Next let $\wt{\bC^{\tms}}$ denote the universal covering of the punctured complex plane 
$\bC^{\times}=\bC\setminus\{0\}$, where the mapping 
$\wt{\bC^{\times}}\ni\wt{Y}\mapsto\log\wt{Y}=\log|\wt{Y}|+i\arg\wt{Y}\in\bC$ is bijective 
(with the range of $\arg\wt{Y}$ being extended over $\bR$). We define for any $X\in\bC$ and 
$\wt{Y}\in\wt{\bC^{\times}}$ the operation 
\begin{align*}
\wt{\bC^{\times}}\ni\wt{Y}\longmapsto\wt{Y}^X
=\exp\{X(\log|\wt{Y}|+i\arg\wt{Y})\}
=|\wt{Y}|^X\exp(iX\log\wt{Y})\in\bC.
\end{align*} 
Let $\wt{e}(\ka)$ for any $\ka\in\bR$ denote the point defined by $\log\wt{e}(\ka)=2\pi i\ka$, and write $\wt{e}(0)=\wt{1}$. Then $\wt{e}(\ka)^{\ga}=e(\ga\ka)$ holds by definition.

It is convenient for describing specific values of $\ps(r,\ga,\ka)$ to introduce the sequence of functions $\cC_k:\bC\times\wt{\bC^{\times}}\ni(X,\wt{Y})\mapsto\cC_k(X,\wt{Y})\in\bC$ $(k=0,1,\ldots)$, defined by the Taylor series expansion (with the variable $Z$ in $\bC$)  
\begin{align*}
\fr{Z\wt{Y}^Xe^{XZ}}{\wt{Y}^1e^Z-1}
&=\sum_{k=0}^{\infty}\fr{\cC_k(X,\wt{Y})}{k!}Z^k
\end{align*}
centered at $Z=0$ (notice that $\wt{Y}^1=|\wt{Y}|\exp(\log\wt{Y})$; this in particular asserts that 
\begin{align*}
\cC_0(X,\wt{Y})=
\begin{cases}
\wt{Y}^X&\quad\text{if $\wt{Y}^1=1$},\\
0           &\quad\text{otherwise},
\end{cases}
\tag{2.15}
\end{align*}
and that $\cC_k(X,\wt{Y})$ reduces when $\wt{Y}=\wt{1}$ (and so $\wt{Y}^X=1$) to the usual Bernoulli polynomial $B_k(X)$. We have shown in \cite[Lemma~3]{katsurada2015} by definition the following reciprocal relations. 
\begin{proposition}
For any integer $k\geq0$ and any $(X,\wt{Y})\in\bC\times\wt{\bC^{\times}}$, we have
\begin{align*}
&\cC_k(1-X,\wt{1}/\wt{Y})=(-1)^k\cC_k(X,\wt{Y}),
\tag{2.16}\\
&\cC_k(0,\wt{1}/\wt{Y})=(-1)^k\cC_k(0,\wt{Y})-\de_{k1}
=\begin{cases}
(-1)^k\cC_k(0,\wt{Y})&\quad\text{if $k\neq1$},\\
-\cC_1(0,\wt{Y})-1&\quad\text{if $k=1$},
\end{cases}
\tag{2.17}
\end{align*}
where $\wt{1}/\wt{Y}\in\wt{\BC^{\times}}$ is the point defined by $|\wt{1}/\wt{Y}|=1/|\wt{Y}|$ 
and $\arg(\wt{1}/\wt{Y})=-\arg\wt{Y}$, and $\de_{kl}$ denotes hereafter Kronecker's symbol. 
\end{proposition}

We proceed to state our second main result, which gives for (1.2) an asymptotic expansion in the ascending order of $\ta$ as $\ta\to0$.
\bg{theorem}
Let $\al$, $\be$, $\mu$ and $\nu$ be real parameters, $z=i\ta\in\mathfrak{H}^+$ with 
$|\arg\ta|<\pi/2$, and set  
\begin{align*}
\cB_1(s,\al,\mu)
&=i\sin(\pi s)\ps(s,-\al,-\mu)
\tag{2.18}\\
&=ie(\al\mu)\fr{(2\pi)^s}{2\vGa(s)}
\bigl\{e^{\pi i(1-s)/2}\ps(1-s,-\mu,\al)\\
&\quad+e^{-\pi i(1-s)/2}\ps(1-s,\mu,-\al)\bigr\},\\
\cB_2(s,\be,\nu)
&=e^{\pi is/2}\ps(s,-\be,-\nu)+e^{-\pi is/2}\ps(s,\be,\nu)
\tag{2.19}\\
&=e(\be\nu)\fr{(2\pi)^s}{\vGa(s)}\ps(1-s,\nu,-\be),
\end{align*}
where the second equalities in (2.18) and (2.19) are valid by (2.8). 
Then for any integer $J\geq0$, in the region $\si>-J$, we have
\begin{align*}
\Fz2{s}{\al,\be}{\mu,\nu}{z}
&=\de(\be)\cB_1(s,\al,\mu)
+\de(\al)\cB_2(s,\be,\nu)\ta^{-s}
\tag{2.20}\\
&\quad+S_J(s;\al,\be;\mu,\nu;z)+R_J(s;\al,\be;\mu,\nu;z)
\end{align*}
in the sector $|\arg\ta|<\pi/2$, where $S_J(s;\al,\be;\mu,\nu;z)$ is the asymptotic series 
of the form  
\begin{align*}
S_J(s;\al,\be;\mu,\nu;z)
&=2\sin(\pi s)\sum_{j=-1}^{J-1}\fr{i^{j+1}(s)_j}{(j+1)!}
\ps(s+j,-\al,-\mu)\cC_{j+1}(\lg\be\rg,\wt{e}(\nu))\ta^j,
\tag{2.21}
\end{align*}
and $R_J(s;\al,\be;\mu,\nu;z)$ is the remainder expressed by the Mellin-Barnes type integrals 
in (5.3) with (5.4) below, and satisfies the estimate 
\begin{align*}
R_J(s;\al,\be;\mu,\nu;z)=O(|\ta|^J)
\tag{2.22}
\end{align*}
as $\ta\to0$ through $|\arg \ta|\leq\pi/2-\et$ with any small $\et>0$. Here the implied $O$-constant depends at most on $s$, $\al$, $\be$, $\mu$, $\nu$, $J$ and $\et$. 
\end{theorem}

Next let $\t1F1{a}{c}{Z}$ and $U(a;c;Z)$ denote Kummer's confluent hypergeometric functions of the 
first and second kind, defined by 
\begin{align*}
\1F1{a}{c}{Z}
&=\sum_{k=0}^{\infty}\fr{(a)_k}{(c)_kk!}Z^k
\qquad(|Z|<+\infty)
\tag{2.23}
\end{align*}
for $(a,c)\in\BC\tms(\BC\setminus\{0,-1,\ldots\})$ (cf.~\cite[p.248, 6.1(1)]{erdelyi1953a}), and 
\begin{align*}
U(a;c;Z)
=\fr{1}{\vGa(a)\{e(a)-1\}}\int_{\infty}^{(0+)}
e^{-Zw}w^{a-1}(1+w)^{c-a-1}dw
\tag{2.24}
\end{align*} 
for $(a,c)\in\BC^2$ and $|\arg Z|<\pi/2$ (the notation in Slater~\cite[p.5, 1.3(1.3.1)]{slater1960} is rather preferable for our later use); this can be continued to the whole sector $|\arg z|<3\pi/2$ by rotating appropriately the path of integration (cf.~\cite[p.273, 6.11.2(9)]{erdelyi1953a}). 
Note that these functions are connected by the relation 
\begin{align*}
\1F1{a}{c}{Z}
&=\vG{c}{c-a}e^{\vep(Z)\pi i\al}U(a;c;Z)
+\vG{c}{a}e^{\vep(Z)\pi i(a-c)}e^Z
\tag{2.25}\\ 
&\quad\tms U(c-a;c;e^{-\vep(Z)\pi i}Z),
\end{align*} 
valid in the sectors $0<|\arg Z|<\pi$ (cf.~\cite[p.~259, 6.7(7)]{erdelyi1953a}\cite[Lemma~2]{katsurada-noda2017}). 

The connection formula (2.25) in fact leads us to extract from the remainder in (2.20) the 
exponentially small order terms $\cS_{1-s}(\pm\al,\pm\nu;\pm\mu,\mp\nu;\wh{q})$ with 
$\wh{q}=e^{-2\pi/\ta}$ (as $\ta\to0$):
\begin{theorem}
Let $\al$, $\be$, $\mu$, and $\nu$ be any real parameters, and write $q=e(z)=e^{-2\pi\ta}$ and 
$\wh{q}=e(-1/z)=e^{-2\pi/\ta}$ for any $z=i\ta\in\fH^+$ with $|\arg\ta|<\pi/2$. Then in the region 
$\si>1-J$ with any $J\geq1$, and in the sectors $0<|\arg\ta|<\pi/2$, we have 
\begin{align*}
&R_J(s,\al,\be;\mu,\nu;z)
\tag{2.26}\\
&\quad=e(\be\nu)\fr{(2\pi/\ta)^s}{\vGa(s)}
\{\cS_{1-s}(\al,\nu;\mu,-\be;\wh{q})
+e^{\vep(\ta)\pi is}\cS_{1-s}(-\al,-\nu;-\mu,\be;\wh{q})\}\\
&\qquad+e(\be\nu)\fr{(-1)^J(s)_J(2\pi/\ta)^s}{\vGa(s)\vGa(1-s)}
R^{\ast}_J(s;\al,\be;\mu,\nu;z).
\end{align*}
Here the expression
\begin{align*}
&R^{\ast}_J(s;\al,\be;\mu,\nu;z)
\tag{2.27}\\
&\quad=\sum_{\bsm \al<m\\ -\nu<n\esm}\fr{e\{(-\al+m)(-\mu)+(\nu+n)(-\be)\}}{(\nu+n)^{1-s}}
F_{s,J}\{2\pi(-\al+m)(\nu+n)/\ta\}\\
&\qquad-e^{\vep(\ta)\pi is}
\sum_{\bsm\al<m\\ \nu<n\esm}\fr{e\{(-\al+m)(-\mu)+(-\nu+n)\be\}}{(-\nu+n)^{1-s}}\\
&\qquad\times F_{s,J}\{2\pi e^{\vep(\ta)\pi i}(-\al+m)(-\nu+n)/\ta\}
\end{align*}
holds with
\begin{align*}
F_{s,J}(Z)=U(s+J;s+J;Z),
\tag{2.28}
\end{align*}
where the right side of (2.27) converges absolutely for $\si>1-J$ with $J\geq1$, and 
provides there the holomorphic continuation of $R^{\ast}_J(s;\al,\be;\mu,\nu;q)$. 
Furthermore for any integers $J\geq1$ and $K\geq0$, in the region $\si>1-J-K$, we 
have 
\begin{align*}
R^{\ast}_J(s;\al,\be;\mu,\nu;z)
&=S^{\ast}_{J,K}(s;\al,\be;\mu,\nu;z)+R^{\ast}_{J,K}(s;\al,\be;\mu,\nu;z)
\tag{2.29}
\end{align*}
where $S^{\ast}_{J,K}(s;\al,\be;\mu,\nu;z)$ is the asymptotic series of the form
\begin{align*}
&S^{\ast}_{J,K}(s;\al,\be;\mu,\nu;z)
\tag{2.30}\\
&\quad=\fr{e(-\be\nu)}{(2\pi e^{-\vep(\ta)\pi i/2})^{s-1}}
\sum_{k=0}^{K-1}\fr{i^{J+k+1}(s+J)_k}{(J+k+1)!}\ps(s+J+k,-\al,-\mu)\\
&\qquad\tms\cC_{J+k+1}(\lg\be\rg,\wt{e}(\nu))(e^{-\vep(\ta)\pi i/2}\ta)^{s+J+k}
\end{align*}  
in the sectors $0<|\arg\ta|<\pi/2$, and $R^{\ast}_{J,K}(s;\al,\be;\mu,\nu;z)$ is the 
remainder expressed by the Mellin-Barnes type integral in (6.24) below, satisfying the 
estimate 
\begin{align*}
R^{\ast}_{J,K}(s;\al,\be;\mu,\nu;z)=O(|\ta|^{\si+J+K})
\tag{2.31}
\end{align*}
as $\ta\to0$ through $\et\leq|\arg\ta|\leq\pi/2-\et$ with any small $\et>0$. Here the 
implied $O$-constant depends at most on $s$, $\al$, $\be$, $\mu$, $\nu$, $J$, $K$ and 
$\et$.
\end{theorem}
\begin{remark}
The formulae (2.26) and (2.29) with (2.30) in fact reveal that the instances of  `exponentially improved 
asymptotics' and `Stokes' phenomena' respectively, which can normally be observed in the 
theory of differential equations in the complex domain, also occur in the present case of  
generalized holomorphic Eisenstein series. 
\end{remark}

We next proceed to state several new variants of Ramanujan's formula for $\ze(2k+1)$. 
\bg{theorem}
Let $q=e(i\ta)=e^{-2\pi\ta}$ and $\wh{q}=e(i/\ta)=e^{-2\pi/\ta}$ for any $\ta\in\bC$ in the sector 
$|\arg \ta|<\pi/2$. Then for any real $\al$, $\be$, $\mu$ and $\nu$, and any integer $k$, supposing that $\al,\be\notin\bbZ$ if $k=1$, we have 
\begin{align*}
&e(\al\mu)\bigl\{\de(\be)\ps(k,-\mu,\al)+\cS_k(\be,-\mu;\nu,\al;q)
+(-1)^{k-1}\cS_k(-\be,\mu;-\nu,-\al;q)\bigr\}
\tag{2.32}\\
&\qquad-(-2\pi)^k\sum_{j=0}^{k+1}
\fr{(-i)^j\cC_j(\lg\al\rg,\wt{e}(\mu))\cC_{k+1-j}(\lg\be\rg,\wt{e}(\nu))}{j!(k+1-j)!}\ta^{k-j}\\
&\quad=e(\be\nu)(-i\ta)^{k-1}\bigl\{\de(\al)\ps(k,\nu,-\be)
+\cS_k(\al,\nu;\mu,-\be;\wh{q})\\
&\qquad+(-1)^{k-1}\cS_k(-\al,-\nu;-\mu,\be;\wh{q})\bigr\},
\end{align*}   
whose variant asserts, upon replacing $(\ta,q)\mapsto(1/\ta,\wh{q})$, that  
\begin{align*}
&e(\be\nu)\bigl\{\de(\al)\ps(k,\nu,-\be)+\cS_k(\al,\nu;\mu,-\be;q)
+(-1)^{k-1}\cS_k(-\al,-\nu;-\mu,\be;q)\bigr\}\\
&\qquad-(-2\pi)^k\sum_{j=0}^{k+1}
\fr{i^j\cC_j(\lg\be\rg,\wt{e}(\nu))\cC_{k+1-j}(\lg\al\rg,\wt{e}(\mu))}{j!(k+1-j)!}\ta^{k-j}\\
&\quad=e(\al\mu)(i\ta)^{k-1}\bigl\{\de(\be)\ps(k,-\mu,\al)+\cS_k(\be,-\mu;\nu,\al;\wh{q})\\
&\qquad+(-1)^{k-1}\cS_k(-\be,\mu;-\nu,-\al;\wh{q})\bigr\}.
\end{align*}
\ed{theorem}
The case $(\mu,\nu)=(0,0)$ of Theorem~4 reduces to the formula for the pairing of $\ze_{\al}(k)$ 
and $\ze_{-\be}(k)$: 
\bg{corollary}
For any real $\al$ and $\be$, and any integer $k$, supposing that $\al,\be\notin\bbZ$ if $k=1$, we have 
\begin{align*}
&\de(\be)\ze_{\al}(k)+\cS_k(\be,0;0,\al;q)
+(-1)^{k-1}\cS_k(-\be,0;0,-\al;q)\\
&\qquad-(-2\pi)^k\sum_{j=0}^{k+1}
\fr{(-i)^jB_j(\lg\al\rg)B_{k+1-j}(\lg\be\rg)
}{j!(k+1-j)!}\ta^{k-j}\\
&\quad=(-i\ta)^{k-1}\bigl\{\de(\al)\ze_{-\be}(k)+\cS_k(\al,0;0,-\be;\wh{q})
+(-1)^{k-1}\cS_k(-\al,0;0,\be;\wh{q})\bigr\}, 
\end{align*}
whose variant asserts that
\begin{align*}
&\de(\al)\ze_{-\be}(k)+\cS_k(\al,0;0,-\be;q)
+(-1)^{k-1}\cS_k(-\al,0;0,\be;q)\\
&\qquad-(-2\pi)^k\sum_{j=0}^{k+1}
\fr{i^jB_j(\lg\be\rg)B_{k+1-j}(\lg\al\rg)}{j!(k+1-j)!}\ta^{k-j}\\
&\quad=(i\ta)^{k-1}\bigl\{\de(\be)\ze_{\al}(k)+\cS_k(\be,0;0,\al;\wh{q})
+(-1)^{k-1}\cS_k(-\be,0;0,-\al;\wh{q})\bigr\}.
\end{align*}
\ed{corollary}
Note that $\cC_k(X,\wt{Y})$ reduces if $X=0$ to $A_k(Y)$ $(k=0,1,\ldots)$, where the rational 
functions $A_k(Y)$ are defined by the Taylor series expansion
\begin{align*}
\fr{Z}{Ye^Z-1}
=\sum_{k=0}^{\infty}\fr{A_k(Y)}{k!}Z^k,
\end{align*}  
centered at $Z=0$. Professor A.~Schinzel kindly let us know (in a private communication with the 
first author) that the coefficients of $A_k(Y)$ are described in terms of Eularien (not Euler's) 
numbers. The case $(\al,\be)=(0,0)$ of Theorem~4 reduces to the formula for the pairing of 
$\ze(k,-\mu)$ and $\ze(k,\nu)$:
\bg{corollary}
For any real $\mu$ and $\nu$, and any integer $k\neq1$ we have
\begin{align*}
&\ze(k,-\mu)+\cS_k(0,-\mu;\nu,0;q)
+(-1)^{k-1}\cS_k(0,\mu;-\nu,0;q)\bigr\}\\
&\qquad-(-2\pi)^k\sum_{j=0}^{k+1}
\fr{(-i)^jA_j(e(\mu))A_{k+1-j}(e(\nu))}{j!(k+1-j)!}\ta^{k-j}\\
&\quad=(-i\ta)^{k-1}\bigl\{\ze(k,\nu)+\cS_k(0,\nu;\mu,0;\wh{q})
+(-1)^{k-1}\cS_k(0,-\nu;-\mu,0;\wh{q})\bigr\},  
\end{align*}
whose variant asserts that
\begin{align*}
&\ze(k,\nu)+\cS_k(0,\nu;\mu,0;q)
+(-1)^{k-1}\cS_k(0,-\nu;-\mu,0;q)\bigr\}\\
&\qquad-(-2\pi)^k\sum_{j=0}^{k+1}
\fr{i^jA_j(e(\nu))A_{k+1-j}(e(\mu))}{j!(k+1-j)!}\ta^{k-j}\\
&\quad=(i\ta)^{k-1}\bigl\{\ze(k,-\mu)+\cS_k(0,-\mu;\nu,0;\wh{q})
+(-1)^{k-1}\cS_k(0,\mu;-\nu,0;\wh{q})\bigr\}.  
\end{align*}
\ed{corollary}
The case $(\be,\nu)=(0,0)$ reduces to the formula for the pairing of $\ps(k,-\mu,\al)$ and 
$\ze(k)$:
\bg{corollary}
For any real $\al$ and $\mu$, and any integer $k\neq1$ we have 
\begin{align*}
&e(\al\mu)\bigl\{\ps(k,-\mu,\al)
+\cS_k(0,-\mu;0,\al;q)
+(-1)^{k-1}\cS_k(0,\mu;0,-\al;q)\bigr\}\\
&\qquad-(-2\pi)^k\sum_{j=0}^{k+1}
\fr{(-i)^j\cC_j(\lg\al\rg,\wt{e}(\mu))B_{k+1-j}}{j!(k+1-j)!}
\ta^{k-j}\\
&\quad=(-i\ta)^{k-1}
\bigl\{\de(\al)\ze(k)+\cS_k(\al,0;\mu,0;\wh{q})
+(-1)^{k-1}\cS_k(-\al,0;\mu,0;\wh{q})\bigr\},
\end{align*} 
whose variant asserts that 
\begin{align*}
&\de(\al)\ze(k)
+\cS_k(\al,0;\mu;0;q)
+(-1)^{k-1}\cS_k(-\al,0,-\mu;0;q)\bigr\}\\
&\qquad-(-2\pi)^k\sum_{j=0}^{k+1}
\fr{i^jB_j\cC_{k+1-j}(\lg\al\rg,\wt{e}(\mu))}{j!(k+1-j)!}
\ta^{k-j}\\
&\quad=e(\al\mu)(i\ta)^{k-1}
\bigl\{\ps(k,-\mu,\al)
+\cS_k(0,-\mu,0;\al,0;\wh{q})\\
&\qquad+(-1)^{k-1}\cS_k(0,\mu,0;0,-\al;\wh{q})\bigr\}.
\end{align*} 
\ed{corollary}
The case $(\al,\nu)=(0,0)$ reduces to the formula for the paring of $\ze(k,-\mu)$ and 
$\ze_{-\be}(k)$: 
\bg{corollary}
For any real $\be$ and $\mu$, and any integer $k\neq1$ we have
\bg{align*}
&\de(\be)\ze(k,-\mu)+\cS_k(\be,-\mu;0,0;q)
+(-1)^{k-1}\cS_k(-\be,\mu;0,0;q)\\
&\qquad-(-2\pi)^k
\sum_{j=0}^{k+1}\fr{(-i)^jA_j(e(\mu))B_{k+1-j}(\lg\be\rg)}{j!(k+1-j)!}\ta^{k-j}\\
&\quad=e(\be\nu)(-i\ta)^{k-1}
\bigl\{\ze_{-\be}(k)+\cS_k(0,0;\mu,-\be;\wt{q})
+(-1)^{k-1}\cS_k(0,0;-\mu,\be;\wh{q}))\bigr\},
\end{align*}
whose variant asserts that 
\bg{align*}
&\ze_{-\be}(k)+\cS_k(0,0;\mu,-\be;q)+(-1)^{k-1}\cS_k(0,0;-\mu,\be;q)\\
&\qquad-(-2\pi)^k
\sum_{j=0}^{k+1}\fr{i^jB_j(\lg\be\rg)A_{k+1-j}(e(\mu))}{j!(k+1-j)!}
\ta^{k-j}\\
&\quad=(i\ta)^{k-1}\bigl\{\de(\be)\ze(k,-\mu)
+\cS_k(\be,-\mu;0,0;\wh{q})
+(-1)^{k-1}\cS_k(-\be,\mu;0,0;\wh{q})\bigr\}.
\end{align*}
\ed{corollary}
Lastly the case $(\al,\be,\mu,\nu)=(0,0,0,0)$ of Theorem~1 reduces to the celebrated formulae 
of Euler and Ramanujan for specific values of $\ze(s)$:
\bg{corollary}
We have the the following formulae{\rm:}
\bg{itemize}
\item[i)] for any integer $k\geq1$, 
\begin{align*}
\ze(2k)=\fr{(-1)^{k-1}(2\pi)^{2k}}{2(2k)!}B_{2k};
\tag{2.33}
\end{align*}
\item[ii)] for any integer $k\neq0$,   
\begin{align*}
&\ze(2k+1)+2\cS_{2k+1}(0,0;0,0;q)
+(2\pi)^{2k+1}
\sum_{j=0}^{k+1}\fr{(-1)^{j}B_{2j}B_{2k+2-2j}}{(2j)!(2k+2-2j)!}\ta^{2k+1-2j}
\tag{2.34}\\
&\quad=(-1)^k\ta^{2k}\{\ze(2k+1)+2\cS_{2k+1}(0,0;0,0;\wh{q})\}.
\end{align*}
\ed{itemize}
\ed{corollary}
\begin{remark}
The formula (2.33) is due to Euler, while (2.34) to Ramanujan. The reader is to be referred to 
\cite[Chap.~14]{berndt1989} for the history and related results on these formulae.    
\end{remark}
\section{Classical Eisenstein series, Weierstra{\ss}' elliptic and allied functions}
We first show in this section several applications of Theorems~1--3 to the classical Eisenstein series of any even integer weight. 

Let 
\begin{align*}
a_k=
\begin{cases}
\ze(1-k)&\quad\text{if $k\neq0$},\\
2\ze'(0)=-\log2\pi&\quad\text{if $k=0$}
\end{cases}
\tag{3.1}
\end{align*}
(cf.~\cite[p.~34, 1.12(18)]{erdelyi1953a}) for any integer $k$, and define for $z\in\fH^+$ 
the classical Eisenstein series $E_k(z)$ (of weight $k\in2\bbZ$) by 
\begin{align*}
E_k(z)=1+\fr{2}{a_k}\sum_{l=1}^{\infty}\fr{l^{k-1}q^l}{1-q^l}
\tag{3.2}
\end{align*} 
with $q=e(z)$ (cf.~\cite[p.105, 4.5(4.5.1)]{berndt2006}). Theorem~1 readily implies upon the second equality in (2.13) and (7.12) below that 
\begin{align*}
E_k(z)=
\begin{cases}
{\dps\fr{(k-1)!}{(2\pi i)^{k}a_{k}}\Fz2{k}{0,0}{0,0}{z}}&\quad\text{if $k\geq2$},\\
{\dps\fr{(2\pi i)^{-k}}{(-k)!a_{k}}F'_{\bbZ^2}(k;0,0;0,0;z)}&\quad\text{if $k\leq0$}
\end{cases}
\tag{3.3}
\end{align*}
for $k\in2\bbZ$, where $F'_{\bbZ^2}(s;0,0;0,0;z)=(\del F_{\bbZ^2}/\del s)(s;0,0;0,0;z)$. Then the combination of Theorems~1,~2 with Theorem~3 in fact yields the following (quasi) modular relations. 
\begin{corollary}
For any $k\in2\bbZ$, we have 
\begin{align*}
E_{k}\Bigl(-\fr{1}{z}\Bigr)
&=z^{k}E_{k}(z)+\fr{1}{a_{k}}\biggl\{\de_{k0}\Bigl(\fr{\pi i}{2}-\log z\Bigr)
\tag{3.4}\\
&\quad+(-2\pi i)^{1-k}\sum_{j=0}^{1-k/2}\fr{B_{2j}B_{2-k-2j}}{(2j)!(2-k-2j)!}z^{1-2j}\biggr\},
\end{align*}
which in particular reduces to
\begin{align*}
E_0\Bigl(-\fr{1}{z}\Bigr)
&=E_0(z)+\fr{1}{\log2\pi}
\Bigl(\log z+\fr{\pi i}{2}+\fr{\pi iz}{6}+\fr{\pi i}{6z}\Bigr),
\tag{3.5}
\end{align*}
\begin{align*}
E_2\Bigl(-\fr{1}{z}\Bigr)
&=z^2E_2(z)+\fr{6z}{\pi i}.
\tag{3.6}
\end{align*}
\end{corollary}

We next proceed to present several applications of Theorems~1--3 to Weierstra{\ss}' elliptic and allied functions.

Let $\bom=(\om_1,\om_2)$ be a fundamental parallelogram with $\im(\om_2/\om_1)>0$. Set $\om_2/\om_1=z$, and choose its branch as $0<\arg z<\pi$. Weierstra{\ss}' elliptic function with periods $\bom=(\om_1,\om_2)$ is defined by 
\begin{align*}
\wp(w\mid\bom)
&=\fr{1}{w^2}
+\sum_{\bsm m,n=-\infty\\ (m,n)\neq(0,0)\esm}^{\infty}
\biggl\{\fr{1}{(w-m\om_1-n\om_2)^2}-\fr{1}{(m\om_1+n\om_2)^2}\biggr\}
\tag{3.7}
\end{align*}
(cf.~\cite[p.328, 13.12(4)]{erdelyi1953b}), while (allied) Weierstra{\ss}' zeta and sigma functions by 
\begin{align*}
\ze(w\mid\bom)
&=\fr{1}{w}+\sum_{\bsm m,n=-\infty\\ (m,n)\neq(0,0)\esm}^{\infty}
\biggl\{\fr{1}{w-m\om_1-n\om_2}+\fr{1}{m\om_1+n\om_2}
+\fr{w}{(m\om_1+n\om_2)^2}\biggr\},\tag{3.8}\\
\si(w\mid\bom)
&=w\prod_{\bsm m,n=-\infty\\ (m,n)\neq(0,0)\esm}^{\infty}
\biggl(1-\fr{w}{m\om_1+n\om_2}\biggr)
\exp\biggl\{\fr{w}{m\om_1+n\om_2}
+\fr12\Bigl(\fr{w}{m\om_1+n\om_2}\Bigr)^2\biggr\}.
\tag{3.9}
\end{align*}
respectively (cf.~\cite[p.329, 13.12 (6); (12)]{erdelyi1953b}). It suffices to treat these functions 
with the normalized periods $\bz=(1,z)$, in view of the relations 
\begin{align*}
\wp(cw\mid c\bom)
=c^{-2}\wp(w\mid\bom),
\quad
\ze(cw\mid c\bom)
=c^{-1}\ze(w\mid\bom),
\quad
\si(cw\mid c\bom)
=c\si(w\mid\bom)
\end{align*}
for $c\in\bC^{\tms}$. One can then see that the limit relation 
\begin{align*}
\wp(w\mid\bz)=\lim_{\bsm s\to2\\ \re s>2\esm}
\bigl\{\Fz2{s}{\al,\be}{0,0}{z}-\Fz2{s}{0,0}{0,0}{z}\bigr\}
\tag{3.10}
\end{align*}
is valid for any $w=\al+\be z\in\BC$ with $(\al,\be)\in\BR^2\setminus\bbZ^2$, since the limit point $s=2$ is located on the boundary of the region where the series in (1.1) converges absolutely.  Theorem~1 can therefore be applied on the right side of (3.10) to show the following expression. 
\begin{corollary}
For any $w=\al+\be z\in\mathbb{C}$ with $(\al,\be)\in\mathbb{R}^2\setminus\mathbb{Z}^2$, 
\begin{align*}
\wp(w\mid\bz)
&=-\fr{\pi^2}{3}E_2(z)+\fr{\de(\be)\pi^2}{\sin^2\pi\al}
-4\pi^2\bigl\{\cS_{-1}(\be,0;0,\al;q)+
\cS_{-1}(-\be,0;0,-\al;q)\bigr\}.
\tag{3.11}
\end{align*} 
\end{corollary}
Let $\wh{\bz}=(1,-1/z)$ be the dual periods of $\bz=(1,z)$. Combining Corollary~4.7 with Theorems~2 and 3, we obtain the classical base change formula
\begin{align*}
\wp\Bigl(\fr{w}{z}\ \Big|\ \wh{\bz}\Bigr)
=z^2\wp(w\mid\bz),
\tag{3.12}
\end{align*}
which clarifies that the functional equation (2.8) and the connection formula (2.25) effect validating  (3.12). We next write the (base) parameter, corresponding to the half period, as $p=e(z/2)=e^{-\pi\ta}$, i.e.~$q=p^2$, and then define the Weierstarssian invariants $e_j(\bz)$ $(j=1,2,3)$ by 
\begin{align*}
e_1(\bz)=\wp\Bigl(\fr12\ \Big|\ \bz\Bigr),
\qquad
e_2(\bz)=\wp\Bigl(\fr{z}{2}\ \Big|\ \bz\Bigr),
\qquad
e_3(\bz)=\wp\Bigl(\fr{1+z}{2}\ \Big|\ \bz\Bigr).
\end{align*} 
Corollary~4.7 then implies the Lambert series expressions of Weierstrassian invariants 
(cf.~\cite[p.72, 4.2(4.46)--(4.48)]{venkatachaliengar2012}):
\begin{align*}
\begin{split}
e_1(\bz)
&=4\pi^2\biggl\{\fr1{6}+4\sum_{l=1}^{\infty}\fr{(2l-1)p^{4l-2}}{1-p^{4l-2}}\biggr\},\\
e_2(\bz)
&=4\pi^2\biggl\{-\fr{1}{12}-2\sum_{l=1}^{\infty}\fr{(2l-1)p^{2l-1}}{1-p^{2l-1}}\biggr\},\\
e_3(\bz)
&=4\pi^2\biggl\{-\fr{1}{12}+2\sum_{l=1}^{\infty}\fr{(2l-1)p^{2l-1}}{1+p^{2l-1}}\biggr\},
\end{split}
\tag{3.13}
\end{align*} 
which further yield a significant relation (cf.~\cite[p.72, 4.2(4.49)]{venkatachaliengar2012})
\begin{align*}
e_1(\bz)+e_2(\bz)+e_3(\bz)=0.
\end{align*}
Furthermore we obtain from (3.12) the period change formulae for Weierstrassian invariants:
\begin{align*}
e_1(\wh{\bz})=z^2e_2(\bz),
\qquad
e_2(\wh{\bz})=z^2e_1(\bz),
\qquad
e_3(\wh{\bz})=z^2e_3(\bz).
\end{align*}

We next consider Weiersta{\ss}' zeta function. An alternative definition of $\ze(w\mid\bz)$ 
asserts 
\begin{align*}
\ze(w\mid\bz)
=\fr{1}{w}-\int_0^w\Bigl\{\wp(u\mid\bz)-\fr{1}{u^2}\Bigr\}du
\tag{3.14}
\end{align*}
(cf.~\cite[p.329, 13.12(6)]{erdelyi1953b}), which is used to integrate the expression in 
(3.11), yielding the following formula. 
\begin{corollary}
For any $w=\al+\be z\in\mathbb{C}$ with $(\al,\be)\in]-1,1[^2\setminus\{(0,0)\}$,
\begin{align*}
\ze(w\mid\bz)
&=\fr{\pi^2}{3}E_2(z)w+\de(\be)\pi\cot\pi\al-(\sgn\be)\pi i
\tag{3.15}\\
&\quad-2\pi i\bigl\{\cS_0(\be,0;0,\al;q)-\cS_0(-\be,0;0,-\al;q)\bigr\}.
\end{align*}
\end{corollary}
Combining Corollary~4.8 with Theorems~2~and~3, we obtain the base change formula
\begin{align*}
\ze\Bigl(\fr{w}{z}\ \Big|\ \wh{\bz}\Bigr)
=z\ze(w\mid\bz),
\tag{3.16}
\end{align*}
which again clarifies (2.8) and (2.25) effect validating (3.16). Weierstra{\ss}' eta invariants are further defined by 
\begin{align*}
\eta_1(\bz)=\ze\Bigl(\fr{1}{2}\ \Big|\ \bz\Bigr),
\quad
\eta_2(\bz)=\ze\Bigl(\fr{z}{2}\ \Big|\ \bz\Bigr),
\quad 
\eta_3(\bz)=\ze\Bigl(-\fr{1+z}{2}\ \Big|\ \bz\Bigr).
\end{align*}
Corollary~4.8 therefore gives the evaluations  
\begin{align*}
\eta_1(\bz)=\fr{\pi^2}{6}E_2(z),
\quad
\eta_2(\bz)=\fr{\pi^2}{6}E_2(z)z-\pi i,
\quad
\eta_3(\bz)=-\fr{\pi^2}{6}E_2(z)(1+z)+\pi i,
\end{align*}
which readily imply the classical Legendre relations (cf.~\cite[p. 329, 13.12(10)]{erdelyi1953b}): 
\begin{align*}
\et_1(\bz)\cdot\fr{z}{2}-\et_2(\bz)\cdot\fr{1}{2}
&=\pi i/2,\\
\et_2(\bz)\cdot\Bigl(-\fr{1+z}{2}\Bigr)-\et_3(\bz)\cdot\fr{z}{2}
&=\pi i/2,\\
\et_3(\bz)\cdot\fr{1}{2}-\et_1(\bz)\cdot\Bigl(-\fr{1+z}{2}\Bigr)
&=\pi i/2.
\end{align*}

We finally consider Weierstra{\ss}' sigma function. An alternative definition of $\si(w\mid\bz)$ asserts  
\begin{align*}
\log\si(w\mid\bz)=\log w+\int_0^w\Bigl\{\ze(u\mid\bz)-\fr{1}{u}\Bigr\}du
\tag{3.17}
\end{align*}
(cf.~\cite[p.329, 13.12(12)]{erdelyi1953b}). We use the customary notation $(Z;q)_{\infty}
=\prod_{l=0}^{\infty}(1-Zq^l)$ for any $Z\in\bC$. Then the expression in (3.15) can be 
integrated to show the following formula.
\begin{corollary}
For any $w=\al+\be z\in\mathbb{C}$ with $(\al,\be)\in]-1,1[^2\setminus\{(0,0)\}$, we have
\begin{align*}
\log\si(w\mid\bz)
&=\fr{\pi^2}{6}E_2(z)w^2+(\sgn\be)\pi i\Bigl(\fr{1}{2}-w\Bigr)+\de(\be)\log(2\sin\pi\al)
\tag{3.18}\\
&\quad-\cS_1(\be,0;0,\al;q)-\cS_1(-\be,0;0,-\al;q)
+2\cS_1(0,0;0,0;q)-\log2\pi,
\end{align*}
whose exponential form asserts
\begin{align*}
\si(w\mid\bz)
&=\exp\biggl\{\fr{\pi^2}{6}E_2(z)w^2+(\sgn\be)\pi i\Bigl(\fr{1}{2}-w\Bigr)\biggr\}(2\sin\pi\al)^{\de(\be)}
\tag{3.19}\\
&\quad\times\fr{(e(\al)q^{\lg\be\rg'};q)_{\infty}(e(-\al)q^{\lg-\be\rg'};q)_{\infty}}{2\pi(q;q)_{\infty}^2}.
\end{align*}
\end{corollary}
Combining Corollary~4.9 with Theorems~2~and~3, we obtain the base change formula
\begin{align*}
\si\Bigl(\fr{w}{z}\ \Big|\ \wh{\bz}\Bigr)
=z^{-1}\si(w\mid\bz),
\tag{3.20}
\end{align*}
which again clarifies that (2.8) and (2.25) effect validating (3.20). 
\section{Derivation of the transformation formula}
We prove Theorem~1 in this section. Prior to the proof, several necessary lemmas are prepared. 
For this we use the bilateral Lerch zeta-functions $\ps_{\mathbb{Z}}^{\pm}(r,\ga,\ka)$, defined by 
\begin{align*}
\ps_{\mathbb{Z}}^{\pm}(r,\ga,\ka)
=\sum_{-\ga\neq k\in\bbZ}\fr{e\{(\ga+k)\ka\}}{(\ga+k)^{r}}
\qquad(\re r>1),
\end{align*}
where the argument of each summand is chosen with $\arg(\ga+k)=\pm\pi$ if $\ga+k<0$, and $\arg(\ga+k)=0$ if $0<\ga+k$.
\begin{lemma}
For any complex $r$, and for any real $\ga$ and $\ka$, we have the relation 
\begin{align*}
\ps^{\pm}_{\mathbb{Z}}(r,\ga,\ka)
=e^{\mp\pi ir}\ps(r,-\ga,-\ka)+\ps(r,\ga,\ka).
\tag{4.1}
\end{align*}
\end{lemma}
\begin{proof}
It is straightforward to see by replacing $k=-h$ for $k<-\ga$ in the defining series that 
\begin{align*}
\ps_{\mathbb{Z}}^{\pm}(r,\ga,\ka)
&=\sum_{\ga<h}\fr{e\{(-\ga+h)(-\ka)\}}{\{-(-\ga+h)\}^r}
+\sum_{-\ga<k}\fr{e\{(\ga+k)\ka\}}{(\ga+k)^r},
\end{align*} 
which concludes (4.1), since $\{-(-\ga+h)\}^{-r}=e^{\mp\pi ir}(-\ga+h)^{-r}$ 
with $\arg(-\ga+h)=0$ in $\ps_{\bbZ}^{\pm}(r,\ga,\ka)$.
\end{proof}
\bg{lemma}
For any real $\ga$ and $\ka$, we have the functional relation
\bg{align*}
\ps_{\mathbb{Z}}^{\pm}(r,\ga,\ka)
&=e(\ga\ka)\fr{(2\pi)^r}{\vGa(r)}e^{\mp\pi i r/2}\ps(1-r,\mp\ka,\pm\ga),
\tag{4.2}
\end{align*}
which shows that $\ps_{\bbZ}^{\pm}(r,\ga,\ka)$ is holomorphic over the whole $r$-plane $\BC$.
\ed{lemma}
\begin{proof}
It is seen from (4.1) that 
\begin{align*}
\ps_{\mathbb{Z}}^{\pm}(r,\ga,\ka)
&=e^{\mp\pi ir/2}
\bigl\{e^{\mp\pi ir/2}\ps(r,-\ga,-\ka)+e^{\pm\pi ir/2}\ps(r,\ga,\ka)\bigr\},
\end{align*}
whose right side is further transformed by (2.8) to conclude (4.2). 
\end{proof}
\begin{lemma}
For any real $\ga$ and $\ka$, and any integer $k\geq0$, we have 
\begin{align*}
\cC_k(\lg-\ga\rg,\wt{e}(-\ka))
&=(-1)^k\cC_k(\lg\ga\rg,\wt{e}(\ka))-\de_{k1}\de(\ga).
\tag{4.3}
\end{align*}
\end{lemma}
\begin{proof}
Consider first the case $\ga\notin\bbZ$, where $\lg-\ga\rg=1-\lg\ga\rg$ holds, and hence (2.16) gives 
\begin{align*}
\cC_k(\lg-\ga\rg,\wt{e}(-\ka))
&=\cC_k(1-\lg\ga\rg,\wt{1}/\wt{e}(\ka))
=(-1)^k\cC_k(\lg\ga\rg,\wt{e}(\ka)).
\end{align*}
Next if $\ga\in\bbZ$, from (2.17), 
\begin{align*}
\cC_k(0,\wt{e}(-\ka))
&=\cC_k(0,\wt{1}/\wt{e}(\ka))
=(-1)^k\cC_k(0,\wt{e}(\ka))-\de_{k1}
\end{align*}
holds. Lemma~3 is thus proved. 
\end{proof}
\begin{lemma}
For any real $\ga$ and $\ka$, and any integer $j\geq0$ we have 
\begin{align*}
\res_{r=1}\ps(r,\ga,\ka)
&=\cC_0(\lg\ga\rg,\wt{e}(\ka))
=e(\lg\ga\rg\ka)\de(\ka),
\tag{4.4}\\
\ps(-j,\ga,\ka)
&=-\fr{\cC_{j+1}(\lg\ga\rg,\wt{e}(\ka))}{j+1}-\de_{j0}\de(\ga).
\tag{4.5}
\end{align*}
\end{lemma}
\begin{proof}
We have shown in \cite[(6.6)--(6.8)]{katsurada2015} almost the same results on $\ps_0(r,\ga,\ka)$ in (2.6), which with (2.10) readily concludes (4.4) and (4.5). 
\end{proof}
\begin{lemma}
For any real $\ga$ and $\ka$, and any integer $j\geq0$ we have
\begin{align*}
\ps_{\bbZ}^{\pm}(-j,\ga,\ka)
&=-\de_{j0}\de(\ga).
\tag{4.6}
\end{align*}
\end{lemma}
\begin{proof}
It is seen from (4.1) and (4.5) that 
\begin{align*}
\ps_{\bbZ}^{\pm}(-j,\ga,\ka)
&=\fr{(-1)^{j+1}\cC_{j+1}(\lg-\ga\rg,\wt{e}(-\ka))}{j+1}
-\fr{\cC_{j+1}(\lg\ga\rg,\wt{e}(\ka))}{j+1}
-2\de_{j0}\de(\ga),
\end{align*}
which with (4.3) concludes (4.6).
\end{proof}
We prepare the necessary vertical estimate for $\ps(r,\ga,\ka)$. 
\begin{lemma}
Define the function
\begin{align*}
\mu(\rh)=
\begin{cases}
1/2-\rh&\quad\text{if $\rh<0$},\\
(1-\rh)/2&\quad\text{if $0\leq\rh\leq1$},\\
0&\quad\text{if $1<\rh$}.
\end{cases}
\end{align*}
The for any real $\ga$ and $\ka$, and any $\vep>0$, we have 
\begin{align*}
\ps(r,\ga,\ka)
&=O\{(|\im r|+1)^{\mu(\re r)+\vep}\}
\tag{4.7}
\end{align*}
on the whole $r$-plane $\BC$ except at $r=1$, where the implied $O$-constant depends at most on 
$\re r$, $\ga$, $\ka$ and $\vep$.
\end{lemma}
\begin{proof}
The same vertical estimate has been shown for $\ph_0(r,\ga,\ka)$ in (2.5) 
(cf.~\cite[Sect.~7.1, Lemma~2]{katsurada2020}); this can readily be transferred to (4.7) 
upon (2.6) and (2.10).
\end{proof}

We are now ready to prove Theorem~1. Suppose temporarily that $\si>2$. The proof starts by splitting the defining series in (1.1) as 
\begin{align*}
F^{\pm}_{\mathbb{Z}^2}(s;\al,\be;\mu,\nu;z)
&=\sump_{m=-\infty}^{\infty}
\Biggl\{\sum_{n<-\be}
+\de(\be)\sum_{n=-\be}^{-\be}+\sum_{-\be<n}\Biggr\}
\fr{e\{(\al+m)\mu+(\be+n)\nu\}}{\{\al+m+(\be+n)z\}^s}
\tag{4.8}\\
&=\vSi_1(s;z)+\vSi^{\pm}_0(s)+\vSi_2(s;z),
\end{align*}
say, where 
\begin{align*}
\vSi_0^{\pm}(s)
&=\de(\be)\sum_{-\al\neq m\in\bbZ}\fr{e\{(\al+m)\mu\}}{(\al+m)^s}
=\de(\be)\ps_{\mathbb{Z}}^{\pm}(s,\al,\mu),
\tag{4.9}
\end{align*}
and further splitting shows
\begin{align*}
\vSi_1(s;z)
&=\sump_{m=-\infty}^{\infty}
\sum_{n<-\be}\fr{e\{(\al+m)\mu+(-\be-n)(-\nu)\}}{\{\al+m+(-\be-n)(-z)\}^s}
\tag{4.10}\\
&=\Biggl\{\sum_{-\al\neq m\in\bbZ}
+\de(\al)\sum_{m=-\al}^{-\al}\Biggr\}\sum_{\be<n}
\fr{e\{(\al+m)\mu+(-\be+n)(-\nu)\}}{\{\al+m+(-\be+n)e^{-\pi i/2}\ta\}^s},
\end{align*}
\begin{align*}
\vSi_2(s;z)
&=\sump_{m=-\infty}^{\infty}
\sum_{-\be<n}\fr{e\{(\al+m)\mu+(\be+n)\nu\}}{\{\al+m+(\be+n)z\}^s}
\tag{4.11}\\
&=\Biggl\{\sum_{-\al\neq m\in\bbZ}
+\de(\al)\sum_{m=-\al}^{-\al}\Biggr\}
\sum_{-\be<n}
\fr{e\{(\al+m)\mu+(\be+n)\nu\}}{\{\al+m+(\be+n)e^{\pi i/2}\ta\}^s}.
\end{align*}
Here we replace $-n$ with $n$ on the second line of (4.10), and set $\mp z=e^{\mp\pi i/2}\ta$ in 
each summand on the right sides of (4.10) and (4.11) respectively. 

Throughout the following, we write $w=u+iv$ with real coordinates $u$ and $v$, and denote by 
$(u)$  the vertical straight path from $u-i\infty$ to $u+i\infty$. Then the  $(m,n)$-sums with 
$-\al\neq m$ on the right sides of (4.10) and (4.11) are further transformed by substituting 
\begin{align*}
\fr{1}{\{\al+m+(\mp\be+n)e^{\mp\pi i/2}\ta\}^s}
&=\fr{1}{2\pi i}\int_{(u_{-1})}\vG{s+w,-w}{s}
\fr{\{(\mp\be+n)e^{\mp\pi i/2}\ta\}^{w}}{(\al+m)^{s+w}}dw
\tag{4.12}
\end{align*}
with a constant $u_{-1}$ satisfying $1-\si<u_{-1}<-1$ into each term, where the argument 
of $\al+m$ is chosen as $\arg(\al+m)=\mp\pi$ if $m<-\al$  (according to the attached 
double signs), and as $\arg(\al+m)=0$ if $-\al<m$; this is obtained by taking 
$Z=(\mp\be+n)e^{\mp\pi i/2}\ta/(\al+m)$ in the Mellin-Barnes formula 
\begin{align*}
\fr{1}{(1+Z)^r}=\fr{1}{2\pi i}\int_{(u)}\vG{r+w,-w}{r}Z^wdw
\qquad(|\arg Z|<\pi)
\end{align*}
with a constant $u$ satisfying $-\re r<u<0$ 
(cf.~\cite[p.289,~14.5,~Corollary]{whittaker-watson1927}), where the choice of $\arg(\al+m)$ 
above confirms the condition $|\arg Z|<\pi$. The second equalities in (4.10) and (4.11), 
by changing the order of summation and integration, then become
\begin{align*}
\vSi_1(s;z)=\vSi_{-}(s;z)
\qquad
\text{and}
\qquad
\vSi_2(s;z)=\vSi_{+}(s;z)
\tag{4.13}
\end{align*}
respectively, where
\begin{align*} 
\vSi_{\pm}(s;z)
&=\fr{1}{2\pi i}\int_{(u_{-1})}\vG{s+w,-w}{s}
\ps_{\mathbb{Z}}^{\pm}(s+w,\al,\mu)\ps(-w,\pm\be,\pm\nu)
\tag{4.14}\\
&\quad\tms(e^{\pm\pi i/2}\ta)^wdw
+\de(\al)\ps(s,\pm\be,\pm\nu)(e^{\pm\pi i/2}\ta)^{-s}.
\end{align*}
Note here that the choice of the functions $\ps^{\pm}_{\mathbb{Z}}(s,\al,\mu)$ in (4.14)  
comes from that of $\arg(\al+m)$ for $m<-\al$ in (4.12), and further that the vertical 
integrals converge absolutely for $|\arg\ta|<\pi/2$, since the integrands are, by (4.1), (4.7) 
and Stirling's formula for $\vGa(s)$ (cf.~\cite[p.~492, A.7(A.34)]{ivic1985}), of order 
$O\{|v|^{c(u_{-1})}e^{-(\pi/2-|\arg\ta|)|v|}\}$ as $v\to\pm\infty$ for some constant 
$c(u_{-1})>0$. We therefore conclude, in view of (1.2), the first equality in (2.13), (4.8)--(4.11) 
and of (4.13), the following lemma. 
\begin{lemma}
For any real $\al$, $\be$, $\mu$ and $\nu$ we have, in the region $\si>2$, 
\begin{align*}
F_{\mathbb{Z}^2}(s;\al,\be;\mu,\nu;z)
&=\de(\be)\cA(s,\al,\mu)+\vSi_{-}(s;z)+\vSi_{+}(s;z),
\tag{4.15}
\end{align*}
where $\vSi_{\pm}(s;z)$ are given by (4.14).
\end{lemma}

Let $u'$ be a constant satisfying $u'<\min(-\si,-1)$. We can then move the path of integration in (4.14) to the left from $(u_{-1})$ to $(u')$, since the integrands are of order $O\{|v|^{c(u)}e^{-(\pi/2-|\arg\ta|)|v|}\}$ as $v\to\pm\infty$ for $u'\leq u\leq u_{-1}$; in passing the residues of the relevant poles are computed by (4.6). One can see that the second terms on the right side of (4.14) cancel out with the residues at $w=-s$ of the integrands respectively, yielding
\begin{align*} 
\vSi_{\pm}(s;z)
&=\fr{1}{2\pi i}\int_{(u')}\vG{s+w,-w}{s}
\ps_{\mathbb{Z}}^{\pm}(s+w,\al,\mu)\ps(-w,\pm\be,\pm\nu)(e^{\pm\pi i/2}\ta)^wdw
\tag{4.16}\\
&=e(\al\mu)\fr{(2\pi e^{\mp\pi i/2})^s}{2\pi i}
\int_{(u')}\vG{-w}{s}\ps(1-s-w,\mp\mu,\pm\al)\\
&\quad\tms\ps(-w,\pm\be,\pm\nu)(2\pi\ta)^wdw,
\end{align*}
where the second equality follows by substituting (4.2) into the first integrands in (4.16). Here the 
initial restriction on $\si$ is relaxed to any $\si\in\mathbb{R}$ at this stage, since $u'$ can be taken appropriately according 
to the location of $s$. 

We now substitute the series representations 
\begin{gather*}
\ps(1-s-w,\mp\mu,\pm\al)
=\sum_{\pm\mu<m}
\fr{e\{(\mp\mu+m)(\pm\al)\}}{(\mp\mu+m)^{1-s-w}},\\
\ps(-w,\pm\be,\pm\nu)
=\sum_{\mp\be<n}\fr{e\{(\pm\be+n)(\pm\nu)\}}{(\pm\be+n)^{-w}},
\end{gather*} 
both of which converge absolutely on the line $\re w=u'$, into the integrands in (4.16), to find upon integrating term-by-term that 
\begin{align*}
\vSi_{\pm}(s;z)
&=e(\al\mu)\fr{(2\pi e^{\mp\pi i/2})^s}{\vGa(s)}
\sum_{\bsm\pm\mu<m\\ \mp\be<n\esm}
\fr{e\{(\mp\mu+m)(\pm\al)+(\pm\be+n)(\pm\nu)\}}{(\mp\mu+m)^{1-s}}\\
&\quad\times q^{(\mp\mu+m)(\pm\be+n)}
\end{align*}  
where the last $(m,n)$-sum equals $\cS_{1-s}(\pm\be,\mp\mu;\pm\al,\pm\nu;q)$ by (2.12); 
this concludes from (4.15) the assertion (2.14) of Theorem~1. 
\section{Derivation of the asymptotic series}
We prove Theorem~2 in this section. For this, suppose temporarily that $\si>2$. The proof 
starts, in view of (2.19), (4.14) and (4.15), from the formula
\begin{align*}
F_{\mathbb{Z}^2}(s;\al,\be;\mu,\nu;z)
&=\de(\be)\cA(s,\al,\mu)+\de(\al)\cB_2(s,\be,\nu)\ta^{-s}
+\vSi_{-}^{\ast}(s;z)+\vSi_{+}^{\ast}(s;z),
\tag{5.1}
\end{align*}
where 
\begin{align*}
\vSi_{\pm}^{\ast}(s;z)
&=\fr{1}{2\pi i}
\int_{(u_{-1})}\vG{s+w,-w}{s}\ps_{\mathbb{Z}}^{\pm}(s+w,\al,\mu)\ps(-w,\pm\be,\pm\nu)
\tag{5.2}\\
&\quad\tms(e^{\pm\pi i/2}\ta)^wdw.
\end{align*} 

Let $J\geq0$ be any integer, and $u_J$ a constant satisfying $J-1<u_J<J$. We can then move the path of  integration in (5.2) to the right from $(u_{-1})$ to $(u_J)$, collecting the residues of the relevant poles at $w=j$ $(j=-1,0,1,\ldots,J-1)$ of the integrands, since it is of order $O\{|v|^{c(u)}e^{-(\pi/2-|\arg\ta|)|v|}\}$ as $v\to\pm\infty$ for $u_{-1}\leq u\leq u_J$. In passing, the sum of the residues of the poles of the integrands in $\vSi_{\pm}^{\ast}(s;z)$ are computed by (4.1) and 
(4.3)--(4.5) as follows. It equals, at $w=-1$, 
\begin{align*}
&(s)_{-1}\ps_{\mathbb{Z}}^{-}(s-1,\al,\mu)(-1)\cC_0(\lg-\be\rg,\wt{e}(-\nu))(-i\ta)^{-1}\\
&\qquad+(s)_{-1}\ps_{\mathbb{Z}}^{+}(s-1,\al,\mu)(-1)\cC_0(\lg\be\rg,\wt{e}(\nu))(i\ta)^{-1}\\
&\quad=i(s)_{-1}\bigl\{-\ps_{\mathbb{Z}}^{-}(s-1,\al,\mu)+\ps_{\mathbb{Z}}^{+}(s-1,\al,\mu)\bigr\}
\cC_0(\lg\be\rg,\wt{e}(\nu))\ta^{-1}\\
&\quad=-2\sin(\pi s)(s)_{-1}\ps(s-1,-\al,-\mu)\cC_0(\lg\be\rg,\wt{e}(\nu)).
\end{align*}
It also equals, at $w=0$, 
\begin{align*}
&-\ps_{\mathbb{Z}}^{-}(s,\al,\mu)
\{-\cC_1(\lg-b\rg,\wt{e}(-\nu))-\de(\be))\}
-\ps_{\mathbb{Z}}^{+}(s,\al,\mu)\{-\cC_1(\lg\be\rg,\wt{e}(\nu))-\de(\be)\}\\
&\quad=\{-\ps^{-}_{\mathbb{Z}}(s,\al,\mu)+\ps^{+}_{\mathbb{Z}}(s,\al,\mu)\}\cC_1(\lg\be\rg,\wt{e}(\nu))
+\de(\be)\ps^{+}_{\mathbb{Z}}(s,\al,\mu)\\
&\quad=-2i\sin(\pi s)\ps(s,-\al,-\mu)\cC_1(\lg\be\rg,\wt{e}(\nu))+\de(\be)\ps_{\mathbb{Z}}^{+}(s,\al,\mu).
\end{align*}
It further equals, at $w=j$ $(j=1,2,\ldots)$, 
\begin{align*}
&-\fr{(-1)^j(s)_j}{j!}\ps_{\mathbb{Z}}^{-}(s+j,\al,\mu)(-1)
\fr{\cC_{j+1}(\lg-\be\rg,\wt{e}(-\nu))}{j+1}(-i\ta)^j\\
&\qquad
-\fr{(-1)^j(s)_j}{j!}\ps_{\mathbb{Z}}^{+}(s+j,\al,\mu)
(-1)\fr{\cC_{j+1}(\lg\be\rg,\wt{e}(\nu))}{j+1}(i\ta)^j\\
&\quad=\fr{(-i)^j(s)_j}{(j+1)!}
\bigl\{-\ps_{\mathbb{Z}}^{-}(s+j,\al,\mu)+\ps_{\mathbb{Z}}^{+}(s+j,\al,\mu)\bigr\}
\cC_{j+1}(\lg\be\rg,\wt{e}(\nu))\ta^j\\
&\quad=-2\sin(\pi s)\fr{i^{j+1}(s)_j}{(j+1)!}\ps(s+j,-\al,-\mu)\cC_{j+1}(\lg\be\rg,\wt{e}(\nu))\ta^j.
\end{align*}
We therefore obtain from (5.1), upon setting 
\begin{align*}
\cB_1(s,\al,\mu)=\cA(s,\al,\mu)-\ps_{\bbZ}^+(s,\al,\mu),
\end{align*}
the expression in (2.20) with (2.21) and 
\begin{align*}
R_J(s;\al,\be;\mu,\nu;z)
=\vSi^{\ast}_{-,J}(s;z)+\vSi^{\ast}_{+,J}(s;z),
\tag{5.3}
\end{align*}
where
\begin{align*}
\vSi^{\ast}_{\pm,J}(s;z)
&=\fr{1}{2\pi i}\int_{(u_J)}\vG{s+w,-w}{s}
\ps_{\mathbb{Z}}^{\pm}(s+w,\al,\mu)
\ps(-w,\pm\be,\pm\nu)(e^{\pm\pi i/2}\ta)^wdw,
\tag{5.4}
\end{align*}
and this confirms the assertion (2.20), in view of (2.13) and (2.18). Here the initial restriction on $\si$ can be relaxed at this stage into $\si>-J$, under which the path can be taken as a straight line $(u_J)$ with $\max(-\si,J-1)<u_J<J$. Moreover the estimate (2.22) can be derived by moving further 
the path of integration in (5.4) from $(u_J)$ to $(u_{J+1})$, yielding
\begin{align*}
R_J(s;\al,\be;\mu,\nu;z)
&=2\sin(\pi s)\fr{i^{J+1}(s)_J}{(J+1)!}\ps(s+J,-\al,-\mu)
\cC_{J+1}(\lg\be\rg,\wt{e}(\nu))\ta^J\\
&\quad+R_{J+1}(s;\al,\be;\mu,\nu;z)
\ll|\ta|^J+|\ta|^{u_{J+1}}\ll|\ta|^J
\end{align*} 
as $\ta\to0$ through $|\arg\ta|\leq\pi/2-\et$ for any small $\et>0$, 
since $J<u_{J+1}$; this concludes (2.22). 
\section{Derivation of an explicit formula for the remainder}
The proof of Theorem~3 is given in this section. 

We first show the assertions (2.26) and (2.27) of Theorem~3. Substituting the expressions on the right sides of (2.8) and (4.1) into the integrand in (5.4), we obtain, after some rearrangements,
\begin{align*}
\vSi^{\ast}_{\pm,J}(s;z)
&=\fr{e(\be\nu)}{2\pi i}\int_{(u_J)}\vG{s+w,-w,1+w}{s}\\
&\quad\tms\bigl\{
e^{\mp\pi i(2s+2w+1)/2)}\ps(s+w,-\al,-\mu)\ps(1+w,-\nu,\be)\\
&\quad+e^{\mp\pi i(2s-1)/2)}\ps(s+w,-\al,-\mu)\ps(1+w,\nu,-\be)\\
&\quad+e^{\mp\pi i/2}\ps(s+w,\al,\mu)\ps(1+w,-\nu,\be)\\
&\quad+e^{\pm\pi i(2w+1)/2}\ps(s+w,\al,\mu)\ps(1+w,\nu,-\be)
\bigr\}\fr{\ta^w}{(2\pi)^{1+w}}dw,
\end{align*}
both of which are summed by (5.3) to give
\begin{align*}
R_J(s;\al,\be;\mu,\nu;z)&=X_1+X_2+X_3,
\tag{6.1}
\end{align*}
say, where 
\begin{align*}
X_1&=-\fr{e(\be\nu)}{2\pi i}\int_{(u_J)}
\vG{-w,1+w}{s,1-s-w}
\ps(s+w,-\al,-\mu)\ps(1+w,-\nu,\be)(\ta/2\pi)^wdw,
\tag{6.2}\\
X_2&=\fr{e(\be\nu)}{2\pi i}\int_{(u_J)}
\vG{s+w,-w,1+w}{s,s,1-s}
\ps(s+w,-\al,-\mu)\ps(1+w,\nu,-\be)
\tag{6.3}\\
&\quad\tms(\ta/2\pi)^wdw,\\
X_3&=\fr{e(\be\nu)}{2\pi i}
\int_{(u_J)}\vG{-w}{s}\ps(s+w,\al,\mu)
\ps(1+w,\nu,-\be)(\ta/2\pi)^wdw.
\tag{6.4}
\end{align*}
We use here the relations $e^{\pi i(s+w+1/2)}+e^{-\pi i(s+w+1/2)}$ $=$ 
$-2\pi/\vGa(1-s-w)\vGa(s+w)$, $e^{-\pi i(s-1/2)}+e^{\pi i(s-1/2)}$ $=$ 
$2\pi/\vGa(s)\vGa(1-s)$, $e^{-\pi i/2}+e^{\pi i/2}$ $=0$ 
and $e^{\pi i(w+1/2)}+e^{-\pi i(w+1/2)}$ $=$ $2\pi/\vGa(-w)\vGa(1+w)$ to modify the resulting 
sums. 

We now proceed to evaluate $X_j$ $(j=1,2,3)$. For this, $X_1$ is first treated. The series 
representations for $\ps(s+w,-\al,-\mu)$ and $\ps(1+w,-\nu,\be)$, both of whose variables 
are in the region of absolute convergence by the choice of $u_J$, are substituted into the integrand in (6.2), to give 
\begin{align*}
X_1
&=-e(\be\nu)\sum_{\bsm\al<m\\ \nu<n\esm}
\fr{e\{(-\al+m)(-\mu)+(-\nu+n)\be\}}{(-\al+m)^s(-\nu+n)}
\tag{6.5}\\
&\quad\tms G_{1,s,J}\{2\pi(-\al+m)(-\nu+n)/\ta\},
\end{align*}
where 
\begin{align*}
G_{1,s,J}(Z)=\fr{1}{2\pi i}\int_{(u_J)}\vG{-w,1+w}{s,1-s-w}Z^{-w}dw
\tag{6.6}
\end{align*}
for $\si>1-J$ with $J\geq1$ and for $|\arg Z|<\pi/2$. 
\begin{lemma}
We have for $|\arg Z|<\pi$, by analytic continuation, 
\begin{align*}
G_{1,s,J}(Z)
=\fr{(-1)^J(s)_J}{\vGa(s)\vGa(2-s-J)}\1F1{1}{2-s-J}{-Z},
\tag{6.7}
\end{align*}  
which is further transformed by the connection formula (2.25), into  
\begin{align*}
G_{1,s,J}(Z)
=Z^se^{-\vep(Z)\pi is}\biggl\{\fr{(-1)^J(s)_J}{\vGa(s)\vGa(1-s)}
F_{s,J}(e^{-\vep(Z)\pi i}Z)-\fr{e^{-Z}}{\vGa(s)}\biggr\}
\tag{6.8}
\end{align*}
in the sectors $0<|\arg Z|<\pi$, where $F_{s,J}(Z)$ is defined by (2.28). 
\end{lemma}
\begin{proof}
We change the variable in (6.6) as $w=J-1-w'$, to obtain 
\begin{align*}
G_{1,s,J}(Z)
=\fr{(-1)^JZ^{1-J}}{2\pi i}
\int_{(u')}\vG{-w',1+w'}{s,2-s-J+w'}Z^{w'}dw',
\tag{6.9}
\end{align*}
where $-1<u'=\re w'=J-1-u_J<\min(\si+J-2,0)$, and use the fact
\begin{align*}
\vGa(w'-J+1)\vGa(-w'+J)=(-1)^J\vGa(-w')\vGa(1+w')
\tag{6.10}
\end{align*}
to modify the integrand of the resulting integral; the right side of (6.9) 
is further evaluated by the Mellin-Barnes formula 
\begin{align*}
\1F1{a}{c}{Z}
&=\fr{1}{2\pi i}\int_{(u)}\vG{a+w,c,-w}{a,c+w}(-Z)^wdw
\end{align*}
for $|\arg(-Z)|<\pi/2$ with a constant $u$ satisfying $-\re a<u<0$ (cf.~\cite[p.256,~6.5(4)]{erdelyi1953a}), to conclude (6.7).

Next the connection formula (2.25) with $-Z=e^{-\vep(Z)\pi i}Z$ instead of $Z$ is applied on the right side of (6.7) to assert, upon using $\vep(-Z)=\vep(e^{-\vep(Z)\pi i}Z)=-\vep(Z)$, that
\begin{align*}
G_{1,s,J}(Z)
&=\fr{(-1)^JZ^{1-J}}{\vGa(s)\vGa(2-s-J)}
\biggl\{\vG{2-s-J}{1-s-J}e^{-\pi i}U(1;2-s-J;e^{-\vep(Z)\pi i}Z)
\tag{6.11}\\
&\quad+\vG{2-s-J}{1}e^{-\vep(Z)\pi i(s+J-1)}e^{-Z}U(1-s-J;2-s-J;Z)\biggr\}
\end{align*} 
for $0<|\arg Z|<\pi$. Here the first term on the right side is further rewritten by the relation
\begin{align*}
U(a;c;Z)=Z^{1-c}U(a-c+1;2-c;Z)
\tag{6.12}
\end{align*}
(cf.~\cite[p.257,~6.6(6)]{erdelyi1953a}), while the second is evaluated by 
noting $U(a;a+1;Z)$ $=$ $Z^{-a}$ for any $a\in\mathbb{C}$ and 
for $|\arg Z|<\pi$, coming from the case $c=a+1$ of (2.24); this with the fact 
$\vGa(1-s-J)=\vGa(1-s)/(-1)^J(s)_J$ concludes (6.8). 
\end{proof} 

We can now substitute the expression in (6.8) with $Z=2\pi(-\al+m)(-\nu+n)/\ta$ into each term on the right side of (6.5), upon noting $\vep(Z)=-\vep(\ta)$ for $0<|\arg\ta|<\pi/2$, to find that 
\begin{align*}
X_1
&=e(\be\nu)e^{\vep(\ta)\pi is}(2\pi/\ta)^s
\tag{6.13}\\
&\quad\tms\biggl[
-\fr{(-1)^J(s)_J}{\vGa(s)\vGa(1-s)}
\sum_{\bsm\al<m\\ \nu<n\esm}
\fr{e\{(-\al+m)(-\mu)+(-\nu+n)\be\}}{(-\nu+n)^{1-s}}\\
&\quad\tms F_{s,J}\{2\pi e^{\vep(\ta)\pi i}(-\al+m)(-\nu+n))/\ta\}\\
&\quad+\fr{1}{\vGa(s)}
\sum_{\bsm\al<m\\ \nu<n\esm}\fr{e\{(-\al+m)(-\mu)+(-\nu+n)\be\}}{(-\nu+n)^{1-s}}
\wh{q}^{(-\al+m)(-\nu+n)}\biggr],
\end{align*}
where the last $(m,n)$-sum equals $\cS_{1-s}(-\al,-\nu;-\mu,\be;\wh{q})$ by (2.12). 

We next treat $X_2$. For this, the series representations for $\ps(s+w,-\al,-\mu)$ and 
$\ps(1+w,\nu,-\be)$, both of whose variables are in the region of absolute convergence, 
are substituted into the integrand in (6.3), and then integrated term-by-term, to assert
\begin{align*}
X_2
&=e(\be\nu)\sum_{\bsm\al<m\\ -\nu<n\esm}
\fr{e\{(-\al+m)(-\mu)+(\nu+n)(-\be)\}}{(-\al+m)^s(\nu+n)}
\tag{6.14}\\ 
&\quad\tms G_{2,s,J}\{2\pi(-\al+m)(\nu+n)/\ta\},
\end{align*}  
where, for $\si>1-J$ and $|\arg Z|<3\pi/2$, 
\begin{align*}
G_{2,s,J}(Z)
=\fr{1}{2\pi i}\int_{(u_J)}\vG{s+w,-w,1+w}{s,s,1-s}Z^{-w}dw,
\tag{6.15}
\end{align*}
which is evaluated by the following lemma. 
\begin{lemma}
For any $\si>1-J$ with $J\geq1$, and for $|\arg Z|<3\pi/2$, we have
\begin{align*}
G_{2,s,J}(Z)
=\fr{(-1)^J(s)_JZ^s}{\vGa(s)\vGa(1-s)}F_{s,J}(Z).
\tag{6.16}
\end{align*}
\end{lemma}
\begin{proof}
Changing the variable in (6.15) as $w=J-1-w'$, we have 
\begin{align*}
G_{2,s,J}(Z)
&=\fr{(-1)^JZ^{1-J}}{2\pi i}
\int_{(u')}\vG{1+w',-w',s+J-1-w'}{s,s,1-s}Z^{w'}dw',
\end{align*}
where $-1<u'=\re w'=-u_J+J-1<\min(\si+J-2,0)$, by the choice of $u_J$, and (6.10) is used to modify the resulting integrand; this is evaluated by the Mellin-Barnes formula
\begin{align*}
U(a;c;Z)
=\fr{1}{2\pi i}\int_{(u)}\vG{a+w,-w,1-c-w}{a,a-c+1}Z^wdw
\tag{6.17}
\end{align*}
for $|\arg Z|<3\pi/2$ with a constant $u$ satisfying $-\re a<u<\min(0,1-\re c)$ 
(cf.~\cite[p.256,~6.5(5)]{erdelyi1953a}), and is further rewritten again by (6.10) to 
conclude (6.16).   
\end{proof}

We can now substitute the expression in (6.16) with $Z=2\pi(-\al+m)(\nu+n)/\ta$ into 
each term on the right side of (6.14), to obtain  
\begin{align*}
X_2
&=e(\be\nu)\fr{(-1)^J(s)_J(2\pi/\ta)^s}{\vGa(s)\vGa(1-s)}
\sum_{\bsm\al<m\\ -\nu<n\esm}
\fr{e\{(-\al+m)(-\mu)+(\nu+n)(-\be)\}}{(\nu+n)^{1-s}}
\tag{6.18}\\
&\quad\tms F_{s,J}\{2\pi(-\al+m)(\nu+n)/\ta\}.
\end{align*}

We lastly treat $X_3$. The series representations for $\ps(s+w,\al,\mu)$ and 
$\ps(1+w,\nu,-\be)$, both of whose variables are in the region of absolute convergence, 
are substituted into the integrand in (6.4), to give
\begin{align*}
X_3
&=e(\be\nu)\sum_{\bsm-\al<m\\ \-\nu<n\esm}
\fr{e\{(\al+m)\mu+(\nu+n)(-\be)\}}{(\al+m)^s(\nu+n)}
G_{3,s,J}\{2\pi(\al+m)(\nu+n)/\ta\},
\tag{6.19}
\end{align*}
with 
\begin{align*}
G_{3,s,J}(Z)
=\fr{1}{2\pi i}\int_{(u_J)}\vG{s+w}{s}Z^{-w}dw
=\fr{Z^s}{\vGa(s)}e^{-Z},
\tag{6.20}
\end{align*} 
for $\si>1-J$ $(J\geq1)$ and $|\arg Z|<\pi/2$, where the last equality follows by changing 
the variable as $w=-s-w'$, and by applying the Mellin inversion formula for $e^{-Z}$, upon noting 
$-J-\si<u'=\re w'=-u_J-\si<\min(-1,1-J-\si)$.
Substituting the expression in (6.20) into each term on the right side of (6.19), we find 
\begin{align*}
X_3
&=e(\be\nu)\fr{(2\pi/\ta)^s}{\vGa(s)}
\sum_{\bsm-\al<m\\ -\nu<n\esm}\fr{e\{(\al+m)\mu+(\nu+n)(-\be)\}}{(\nu+n)^{1-s}}
\wh{q}^{(\al+m)(\nu+n)},
\tag{6.21}
\end{align*} 
where the last $(m,n)$-sum equals $\cS_{1-s}(\al,\nu;\mu,-\be;\wh{q})$ by (2.12).

We thus sum up the results (6.13), (6.18) and (6.21), in view of (6.1), to 
conclude the assertions (2.26) and (2.27) of Theorem~3.

We next proceed to prove the assertions (2.29)--(2.31). It follows from (2.28) and  (6.17) 
that 
\begin{align*}
F_{s,J}(Z)
&=\fr{1}{2\pi i}
\int_{(u_J)}\vG{s+J+w,-w,1-s-J-w}{s+J}Z^wdw
\end{align*}
for $|\arg Z|<3\pi/2$ with a constant $u_J$ satisfying $-\si-J<u_J<\min(0,1-\si-J)$; 
this is substituted into each term on the right side of (2.27), and then the order of the 
$(m,n)$-sum and the $w$-integral is interchanged, to show that 
\begin{align*}
&R^{\ast}_J(s;\al,\be;\mu,\nu;z)
\tag{6.22}\\
&\quad=\fr{1}{2\pi i}\int_{(u_J)}\vG{s+J+w,-w,1-s-J-w}{s+J}\ps(-w,-\al,-\mu)\\
&\qquad\tms\bigl\{\ps(1-s-w,\nu,-\be)
-e^{\vep(\ta)\pi i(s+w)}\ps(1-s-w,-\nu,\be)\bigr\}(2\pi/\ta)^wdw\\
&\quad=\fr{e(-\be\nu)}{(2\pi e^{-\vep(\ta)\pi i})^{s-1}}
\fr{1}{2\pi i}\int_{(u_J)}\vG{s+J+w,-w,1-s-J-w}{s+w,1-s-w}\\
&\qquad\tms\ps(s+w,-\al,-\mu)\ps(s+w,\vep(\ta)\be,\vep(\ta)\nu)(e^{\vep(\ta)\pi i/2}/\ta)^wdw,
\end{align*}  
where the integrand on the rightmost side is derived by the following lemma.
\begin{lemma}
For any real $\be$ and $\nu$, and in the sectors $0<|\arg\ta|<\pi/2$, we have 
\begin{align*}
&\ps(1-s-w,\nu,-\be)-e^{\vep(\ta)\pi i(s+w)}\ps(1-s-w,-\nu,\be)
\tag{6.23}\\
&\quad=e(-\be\nu)\fr{(2\pi e^{-\vep(\ta)\pi i/2})^{1-s-w}}{\vGa(1-s-w)}
\ps(s+w,\vep(\ta)\be,\vep(\ta)\nu).
\end{align*}
\end{lemma}
\begin{proof}
We use the functional equation (2.8) on the left side of (6.23) to see that it becomes 
\begin{align*}
&e(-\be\nu)\fr{\vGa(s+w)}{(2\pi)^{s+w}}e^{\vep(\ta)\pi i(s+w)/2}(2i)
\bigl[\sin\{(1-\vep(\ta))\pi(s+w)/2\}\\
&\qquad\tms\ps(s+w,-\be,-\nu)
-\sin\{(1+\vep(\ta))\pi(s+w)/2\}\ps(s+w,\be,\nu)\bigr]\\
&\quad=e(-\be\nu)\fr{\vGa(s+w)}{(2\pi)^{s+w}}
e^{\vep(\ta)(s+w-1)\pi i/2}2\sin\{\pi(s+w)\}\ps(s+w,\vep(\ta)\be,\vep(\ta)\nu),
\end{align*}
which further equals the right side of (6.23). 
\end{proof}

We now prove (2.29)--(2.31). Let $K\geq0$ be an integer, and $u_{J,K}$ a constant 
satisfying $-\si-J-K<u_{J,K}<\min(1-\si-J-K,0)$. We can then move the path of 
integration in (6.22) from $(u_J)$ to $(u_{J,K})$, and this gives the expression (2.29) 
with 
\begin{align*}
&R^{\ast}_{J,K}(s;\al,\be;\mu,\nu;z)
\tag{6.24}\\
&\quad=\fr{e(-\be\nu)}{(2\pi e^{-\vep(\ta)\pi i/2})^{s-1}}
\fr{1}{2\pi i}\int_{(u_{J,K})}\vG{s+J+w,-w,1-s-J-w}{s+J,1-s-w}\\
&\qquad\tms\ps(-w,-\al,-\mu)\ps(s+w,\vep(\ta)\be,\vep(\ta)\nu)
(e^{\vep(\ta)\pi i/2}/\ta)^wdw,
\end{align*}
which is estimated similarly to (2.22), concluding (2.31). The proof of Theorem~3 is thus complete.

\section{Derivation of variants of Ramanujan's formula}
We prove Theorem~4 in this section; the four cases when i) $k\leq-2$; ii) $k\geq1$; iii) $k=0$; and iv) $k=-1$ are separately treated.
\begin{proof}[Proof of Case i) $k\leq-2$] 
Prior to the proof, the following Lemmas~11~and~12 are prepared. 
\begin{lemma}
For any integer $k\leq0$ we have 
\begin{align*}
\ps(k,\mu,-\al)
&=(-1)^{1-k}\ps(k,-\mu,\al)-\de_{k0}\de(\mu).
\tag{7.1}
\end{align*}
\end{lemma}
\begin{proof}
It follows from (4.3) and (4.5) that, for any integer $k\leq0$, 
\begin{align*}
\ps(k,\mu,-\al)
&=-\fr{1}{1-k}
\bigl\{(-1)^{1-k}\cC_{1-k}(\lg-\mu\rg,\wt{e}(\al))-\de_{k0}\de(\mu)\bigr\}
-\de_{k0}\de(\mu)\\
&=(-1)^{1-k}\Bigl\{-\fr{\cC_{1-k}(\lg-\mu\rg,\wt{e}(\al))}{1-k}\Bigr\},
\end{align*}
which again with (4.5) concludes (7.1).
\end{proof}
\begin{lemma}
For any integer $k\leq0$ we have 
\begin{align*}
\cA(1-k,\al,\mu)
&=e(\al\mu)\fr{(-2\pi i)^{1-k}}{(-k)!}
\Bigl\{\ps(k,-\mu,\al)+\fr{1}{2}\de_{k0}\de(\mu)\Bigr\},
\tag{7.2}\\
\cB_1(1-k,\al,\mu)
&=-e(\al\mu)\fr{(2\pi i)^{1-k}}{2(-k)!}\de_{k0}\de(\mu),
\tag{7.3}\\
\cB_2(1-k,\be,\nu)
&=e(\be\nu)\fr{(2\pi)^{1-k}}{(-k)!}\ps(k,\nu,-\be).
\tag{7.4}
\end{align*}
\end{lemma}
\begin{proof}
It follows from the second equality in (2.13) that 
\begin{align*}
\cA(1-k,\al,\mu)
&=e(\al\mu)\fr{(2\pi i)^{1-k}}{2(-k)!}
\bigl\{\ps(k,\mu,-\al)+(-1)^{1-k}\ps(k,-\mu,\al)\bigr\},
\end{align*}
in which (7.1) is substituted to conclude (7.2). Next the second equality in (2.18) shows 
\begin{align*}
\cB_1(1-k,\be,\nu)
&=e(\al\mu)\fr{(2\pi i)^{1-k}}{2(-k)!}
\bigl\{(-1)^k\ps(k,-\mu,\al)+\ps(k,\mu,-\al)\bigr\},
\end{align*}
which with (7.1) concludes (7.3). Lastly the second equality in (2.19) readily implies (7.4)
\end{proof}
The case $s=1-k$ $(k=0,-1,\ldots)$ of (2.14) and (7.2) yields the following formula.
\begin{lemma}
For any integer $k\leq0$ we have 
\begin{align*}
\Fz2{1-k}{\al,\be}{\mu,\nu}{z}
&=e(\al\mu)\fr{(-2\pi i)^{1-k}}{(-k)!}
\Bigl\{\ps(k,-\mu,\al)+\fr{1}{2}\de_{k0}\de(\mu)\Bigr\}
\tag{7.5}\\
&\quad+e(\al\mu)\fr{(-2\pi i)^{1-k}}{(-k)!}
\bigl\{\cS_k(\be,-\mu;\nu,\al;q)
\\
&\quad+(-1)^{k-1}\cS_k(-\be,\mu;-\nu,-\al;q)\bigr\}.
\end{align*}
\end{lemma}
We are now ready to prove Case i) $k\leq-2$. For this the following lemma 
is further shown. 
\begin{lemma}
For any integer $k\leq-2$ we have 
\begin{align*}
\Fz2{1-k}{\al,\be}{\mu,\nu}{z}
&=\de(\al)e(\be\nu)\fr{(2\pi/\ta)^{1-k}}{(-k)!}\ps(k,\nu,-\be)
+e(\be\nu)\fr{(2\pi/\ta)^{1-k}}{(-k)!}
\tag{7.6}\\
&\quad\tms\bigl\{\cS_k(\al,\nu;\mu,-\be;\wh{q})
+(-1)^{k-1}\cS_k(-\al,\nu;-\mu,\be;\wh{q})\bigr\}.
\end{align*}
\end{lemma}
\begin{proof}
We apply Theorems~2~and~3, for which $J$ is to be taken as $1-k=\si>1-J$, being 
fulfilled for any $J\geq0$ in this case. The possible poles of the $j$-sum in (2.21) 
(without the factor  $2\sin(\pi s)$) come from the term with $j=-1$ (asserting that 
$s=1$ and $s=2$ are these poles), and with $j=0$ (asserting that $s=1$ is the 
possible pole), while other terms with $1\leq j<J$ are all holomorphic, since 
$(s)_j=(s)_{j-1}(s+j-1)$ and $(s+j-1)\ps(s+j,-\al,-\mu)$ is holomorphic at $s=1-j$ for 
$j\geq1$. The $j$-sum in (2.21) is hence holomorphic at $s=1-k$ $(k\leq-2)$; this 
shows that $S_J(s;\al,\be;\mu,\nu;z)$ vanishes at $s=1-k$ for any $k\leq-2$ by 
$\sin(\pi s)|_{s=1-k}=0$. 

Moreover, since the factor $R_J^{\ast}(s;\al,\be;\mu,\nu;z)$ 
in (2.26) is, by (2.27), holomorphic in the region $\si>1-J$, the last term on the right 
side of (2.26) also vanishes by (7.12) below, and hence (2.26) becomes, for any $k\leq0$,  
\begin{align*}
R_J(1-k;\al,\be;\mu,\nu;z)
&=e(\be\nu)\fr{(2\pi/\ta)^{1-k}}{(-k)!}
\bigl\{\cS_k(\al,\nu;\mu,-\be;\wh{q})
\tag{7.7}\\
&\quad+(-1)^{k-1}\cS_k(-\al,-\nu;-\mu,\be;\wh{q})\bigr\}.
\end{align*}
We therefore obtain (7.6) from (2.20), (7.3), (7.4) and (7.7). 
\end{proof} 
The assertion (2.32) for Case i) $k\leq-2$ is thus concluded by equating the right sides of (7.5) and (7.6), and then by multiplying both sides by $(-2\pi i)^{k-1}(-k)!$.   
\end{proof}
\begin{proof}[Proof of Case ii) $k\geq1$] 
Throughout the following, the primes on $F_{\bbZ^2}$, $\cA$, $\cB_j$ $(j=1,2)$ and $R_J$ indicate 
the partial differentiation $\del/\del s$ respectively. Prior to the proof, the following Lemmas~15--18 are prepared. 
\begin{lemma}
Let $k\geq1$ be any integer, and suppose further that $\al,\be\notin\bbZ$ if $k=1$. Then we have 
\begin{align*}
\cA'(1-k,\al,\mu)
&=\fr{1}{2}e(\al\mu)
(-2\pi i)^{1-k}(k-1)!\bigl\{\ps(k,\mu,-\al)+(-1)^{k-1}\ps(k,-\mu,\al)\bigr\},
\tag{7.8}\\
\cB_1'(1-k,\al,\mu)
&=-\fr{1}{2}e(\al\mu)(2\pi i)^{1-k}(k-1)!
\bigl\{\ps(k,-\mu,\al)+(-1)^k\ps(k,\mu,-\al)\bigr\},
\tag{7.9}\\
\cB_2(1-k,\be,\nu)&=0,
\tag{7.10}\\
\cB'_2(1-k,\be,\nu)
&=e(\be\nu)(-2\pi)^{1-k}(k-1)!\ps(k,\nu,-\be).
\tag{7.11}
\end{align*}
\end{lemma}
\begin{proof}
We differentiate both sides of the second equality in (2.13), upon noting
\begin{align*}
\Bigl(\fr{1}{\vGa}\Bigr)(-h)=0
\qquad\text{and}
\qquad
\Bigl(\fr{1}{\vGa}\Bigr)'(-h)=(-1)^hh!
\qquad(h=0,1,\ldots),
\tag{7.12}
\end{align*}
to find that 
\begin{align*}
\cA'(1-k,\al,\mu)
&=\fr{1}{2}e(\al\mu)(-2\pi)^{1-k}(k-1)!
\bigl\{i^{1-k}\ps(k,\mu,-\al)+(-i)^{1-k}\ps(k,-\mu,\al)\bigr\},
\end{align*} 
which concludes (7.8). The assertion (7.9) follows similarly from the second equality in (2.18) by using (7.12), while (7.10) and (7.11) from the second equality in (2.19) again by (7.12). 
\end{proof}
\begin{lemma}
Let $k\geq1$ be any integer, and suppose further that $\al,\be\notin\bbZ$ if $k=1$. Then we have 
\begin{align*}
&F'_{\bbZ^2}(1-k;\al,\be;\mu,\nu;z)
\tag{7.13}\\ 
&\quad=\fr{1}{2}e(\al\mu)(2\pi i)^{1-k}(k-1)!
\bigl\{(-1)^{k-1}\ps(k,\mu,-\al)
+\ps(k,-\mu,\al)\bigr\}\\
&\qquad+e(\al\mu)(2\pi i)^{1-k}(k-1)!
\bigl\{\cS_k(\be,-\mu;\nu,\al;q)
+(-1)^{k-1}\cS_k(-\be,\mu;-\nu,-\al;q)\bigr\}.
\end{align*}
\end{lemma}
\begin{proof}
It follows from the case $s=1-k$ $(k=1,2,\ldots)$ of (the differentiated form of) (2.14) that 
\begin{align*}
F'_{\bbZ^2}(1-k;\al,\be;\mu,\nu;z)
&=\de(\be)\cA'(1-k,\al,\mu)+e(\al\mu)(-2\pi)^{1-k}(k-1)!\\
&\quad\tms\bigl\{(-i)^{1-k}\cS_k(\be,-\mu;\nu,\al;q)
+i^{1-k}\cS_k(-\be,\mu;-\nu,-\al;q)\bigr\},
\end{align*}
which with (7.8) concludes (7.13).
\end{proof}
\begin{lemma}
For any integers $j\geq0$ and $k\geq1$, we have 
\begin{align*}
&\res_{s=1}\ps(s,-\al,-\mu)=\cC_0(\lg\al\rg,\wt{e}(\mu))
\tag{7.14}\\
&\ps(-j,-\al,-\mu)=\fr{(-1)^j\cC_{j+1}(\lg\al\rg,\wt{e}(\mu))}{j+1},
\tag{7.15}\\
&(s)_k\ps(s+k,-\al,-\mu)\bigr|_{s=1-k}=(-1)^{k-1}(k-1)!\cC_0(\lg\al\rg,\wt{e}(\mu)).
\tag{7.16}
\end{align*}
\end{lemma}
\begin{proof}
The assertions (7.14) and (7.15) follow from (4.4) and (4.5) respectively by using (4.3), while the left side of (7.16) equals 
\begin{align*}
\lim_{\vep\to0}(1-k+\vep)_{k-1}\vep\ps(1+\vep,-\al,-\mu)
=(1-k)_{k-1}\res_{s=1}\ps(s,-\al,-\mu),
\end{align*}
which with (7.14) concludes (7.16).
\end{proof}
\begin{lemma}
Let $k\geq1$ be any integer, and suppose further that $\al,\be\notin\bbZ$ if $k=1$. Then we have 
\begin{align*}
&F_{\bbZ^2}'(1-k;\al,\be;\mu,\nu;z)
\tag{7.17}\\
&\quad=-\fr{1}{2}\de(\be)e(\al\mu)(2\pi i)^{1-k}(k-1)!
\bigl\{\ps(k,-\mu,\al)+(-1)^k\ps(k,\mu,-\al)\bigr\}\\
&\qquad+\de(\al)e(\be\nu)(-2\pi/\ta)^{1-k}(k-1)!\ps(k,\nu,-\be)\\
&\qquad+2\pi\sum_{j=-1}^k\fr{i^{j+1}(k-1)!}{(j+1)!(k-j)!}
\cC_{k-j}(\lg\al\rg,\wt{e}(\mu))\cC_{j+1}(\lg\be\rg,\wt{e}(\nu))\ta^j\\
&\qquad+e(\be\nu)(-2\pi/\ta)^{1-k}(k-1)!
\bigl\{\cS_k(\al,\nu;\mu,-\be;\wh{q})\\
&\qquad+(-1)^{k-1}\cS_k(-\al,-\nu;-\mu,\be;\wh{q})\bigr\}.
\end{align*}
\end{lemma}
\begin{proof}
It suffices to suppose $1-k=\si<1-J$, i.e. $k<J$ for applying Theorems~2~and~3 at $s=1-k$. 
We first differentiate both sides of (2.21) with respect to $s$, and then set $s=1-k$, to 
find from $\sin(\pi s)|_{s=1-k}=0$ and $\{\sin(\pi s)\}'|_{s=1-k}=(-1)^{k-1}\pi$ that 
\begin{align*}
&S'_J(1-k;\al,\be;\mu,\nu;z)
\tag{7.18}\\
&\quad=2\pi(-1)^{k-1}\Biggl\{\sum_{\bsm j=-1\\ j\neq k\esm}^{J-1}\fr{i^{j+1}(1-k)_j}{(j+1)!}
\ps(1-k+j,-\al,-\mu)\cC_{j+1}(\lg\be\rg,\wt{e}(\nu))\ta^j\\
&\qquad+\fr{i^{k+1}(s)_k}{(k+1)!}\ps(s+k,-\al,-\mu)
\biggr|_{s=1-k}\cC_{k+1}(\lg\be\rg,\wt{e}(\nu))\ta^k\Biggr\},
\end{align*}
where the right side is further transformed by (7.15), (7.16), and by the fact that 
$(1-k)_j$ equals $(-1)^j(k-1)!/(k-j-1)!$ or $0$, according to $-1\leq j\leq k-1$ or $k\leq j$; 
this therefore becomes eventually the $j$-sum on the right side of (7.17). We next 
differentiate both sides of (2.26) to set $s=1-k$. Since $(\del/\del s)
\{(s)_J/\vGa(s)\}|_{s=1-k}=0$ by (7.12) and $(1-k)_J=0$, further implying  
\begin{align*}
\fr{\del}{\del s}\fr{(s)_J(2\pi/\ta)^s}{\vGa(s)\vGa(1-s)}
S^{\ast}_J(s;\al,\be;\mu,\nu;z)\biggr|_{s=1-k}=0
\end{align*}
for $J>k\geq1$, we find, again by (7.12) and $\vep(\ta)=\pm1$, that
\begin{align*}
R_J'(1-k;\al,\be;\mu,\nu;z)
&=e(\be\nu)(-2\pi/\ta)^{1-k}(k-1)!
\bigl\{\cS_k(\al,\nu;\mu,-\be;\wh{q})
\tag{7.19}\\
&\quad+(-1)^{k-1}\cS_k(-\al,-\nu;-\mu,\be;\wh{q})\bigr\},
\end{align*}
giving the last term on the right side of (7.17). Thus the differentiated form of (2.20) with (2.26), 
upon (7.8)--(7.12), (7.18) and (7.19), concludes the assertion (7.17). 
\end{proof}

We now proceed to prove Case ii) $k\geq2$. Equating the right sides of (7.13) and (7.17), cancelling out the factor $(1/2)e(\al\mu)(2\pi i)^{1-k}(k-1)!\ps(k,\mu,-\al)$ from both sides, and then multiplying 
the resulting form by $(2\pi i)^{k-1}/(k-1)!$, we obtain 
\begin{align*}
&\de(\be)e(\al\mu)\ps(k,-\mu,\al)
+e(\al\mu)\bigl\{\cS_k(\be,-\mu;\nu,\al;q)+(-1)^{k-1}\cS_k(-\be,\mu;-\nu,-\al;q)\bigr\}\\
&\quad
=\de(\al)e(\be\nu)(-i\ta)^{k-1}\ps(k,\nu,-\be)
+(2\pi)^k\sum_{j=-1}^k\fr{i^{j+k}
\cC_{k-j}(\lg\al\rg,\wt{e}(\mu))\cC_{j+1}(\lg\be\rg,\wt{e}(\nu))}{(k-j)!(j+1)!}\ta^j\\
&\qquad+e(\be\nu)(-i\ta)^{k-1}
\bigl\{\cS_k(\al,\nu;\mu,-\be;\wh{q})+(-1)^{k-1}\cS_k(-\al,-\nu;-\mu,\be;\wh{q})\bigr\},
\end{align*}
which concludes (2.32) for Case ii) $k\geq2$, after changing the summation index as $j\mapsto k-j$.
\end{proof}
\begin{proof}[Proof of Case iii) $k=-1$]
Prior to the proof, the following Lemmas~19~and~20 are prepared.
\begin{lemma}
We have 
\begin{align*}
\sin(\pi s)(s)_{-1}\ps(s-1,-\al,-\mu)\bigr|_{s=1}
&=-\pi\cC_1(\lg\al\rg,\wt{e}(\mu)),
\tag{7.20}\\
\sin(\pi s)(s)_{-1}\ps(s-1,-\al,-\mu)\bigr|_{s=2}
&=\pi\cC_0(\lg\al\rg,\wt{e}(\mu)),
\tag{7.21}\\
\sin(\pi s)(s)_0\ps(s,-\al,-\mu)\bigr|_{s=1}
&=-\pi\cC_0(\lg\al\rg,\wt{e}(\mu)).
\tag{7.22}
\end{align*}
\end{lemma}  
\begin{proof}
The assertion (7.20) follows from (7.15), while (7.21) and (7.22) from (7.14).
\end{proof}
\begin{lemma}
We have 
\begin{align*}
&\Fz2{2}{\al,\be}{\mu,\nu}{z}
\tag{7.23}\\
&\quad=\de(\al)e(\be\nu)(2\pi/\ta)^2\ps(-1,\nu,-\be)
+2\pi\cC_0(\lg\al\rg,\wt{e}(\mu))\cC_0(\lg\be\rg,\wt{e}(\nu))\ta^{-1}\\
&\qquad e(\be\nu)(2\pi/\ta)^2
\bigl\{\cS_{-1}(\al,\nu;\mu,-\be;\wh{q})
+\cS_{-1}(-\al,-\nu;-\mu,\be;\wh{q})\bigr\}.
\end{align*}
\end{lemma}
\begin{proof}
It follows from (2.21), in view of (7.21), that
\begin{align*}
S_J(2;\al,\be;\mu,\nu;z)
&=2\pi\cC_0(\lg\al\rg,\wt{e}(\mu))\cC_0(\lg\be\rg,\wt{e}(\nu))\ta^{-1}
\end{align*}
for $J\geq0$, which upon (2.20), together with (7.3), (7.4) and (7.7), concludes (7.23).    
\end{proof} 

We now proceed to prove Case iii) $k=-1$. Equating the right sides of (7.5) (with 
$k=-1$) and (7.23), we obtain 
\begin{align*}
&e(\al\mu)(-2\pi i)^2\ps(-1,\mu,-\al)\\
&\qquad+e(\al\mu)(-2\pi i)^2\bigl\{\cS_{-1}(\be,-\mu;\nu,\al;q)
+\cS_{-1}(-\be,\mu;-\nu,-\al;q)\bigr\}\\
&\quad=\de(\al)e(\be\nu)(2\pi)^2\ps(-1,\nu,-\be)\ta^{-2}
+2\pi\cC_0(\lg\al\rg,\wt{e}(\mu))\cC_0(\lg\be\rg,\wt{e}(\nu))\ta^{-1}\\
&\qquad+e(\be\nu)(2\pi/\ta)^2
\bigl\{\cS_{-1}(\al,\nu;\mu,-\be;\wh{q})+\cS_{-1}(-\al,-\nu;-\mu,\be;\wh{q})\bigr\},
\end{align*}
in which the factor $(-2\pi i)^{-2}$ is multiplied by both sides, after some rearrangements, 
to conclude (2.32) for $k=-1$. 
\end{proof}
\begin{proof}[Proof of Case iv) $k=0$.]
Prior to the proof, we prepare the following Lemma~21. 
\begin{lemma}
We have 
\begin{align*}
&\Fz2{1}{\al,\be}{\mu,\nu}{z}
\tag{7.24}\\
&\quad=-\de(\be)e(\al\mu)(\pi i)\de(\mu)+\de(\al)e(\be\nu)(2\pi)\ps(0,\nu,-\be)\ta^{-1}\\
&\qquad-2\pi\cC_1(\lg\al\rg,\wt{e}(\mu))\cC_0(\lg\be\rg,\wt{e}(\nu))\ta^{-1}
-2\pi i\cC_0(\lg\al\rg,\wt{e}(\mu))\cC_1(\lg\be\rg,\wt{e}(\nu))\\
&\qquad+e(\be\nu)(2\pi/\ta)
\bigl\{\cS_0(\al,\nu;-\mu,-\be;\wh{q})-\cS_0(-\al,-\nu;-\mu,\be;\wh{q})\bigr\}.
\end{align*}
\end{lemma}
\begin{proof}
It follows from (2.21), in view of (7.20) and (7.22), that
\begin{align*}
S_J(1;\al,\be;\mu,\nu;z)
&=-2\pi\big\{\cC_1(\lg\al\rg,\wt{e}(\mu))\cC_0(\lg\be\rg,\wt{e}(\nu))\ta^{-1}+i\cC_0(\lg\al\rg,\wt{e}(\mu))\cC_1(\lg\be\rg,\wt{e}(\nu))\big\}
\end{align*}
for $J\geq0$, which upon (2.20), together with (7.3), (7.4) and (7.7), concludes (7.24). 
\end{proof}
We now proceed to the proof of Case iv) $k=0$. Equating the right sides of (7.5) (with $k=0$) and (7.24), we obtain 
\begin{align*}
&\de(\be)e(\al\mu)(-2\pi i)
\Bigl\{\ps(0,-\mu,\al)+\fr{1}{2}\de(\mu)\Bigr\}\\
&\qquad+e(\al\mu)(-2\pi i)\bigl\{\cS_0(\be,-\mu;\nu,\al;q)
-\cS_0(-\be,\mu;-\nu,\al;q)\bigr\}\\
&\quad=-\de(\be)e(\al\mu)\pi i\de(\mu)+\de(\al)e(\be\nu)(2\pi)\ps(0,\nu,-\be)\ta^{-1}\\
&\qquad+2\pi\bigl\{i\cC_0(\lg\al\rg,\wt{e}(\mu))\cC_1(\lg\be\rg,\wt{e}(\nu))
+\cC_1(\lg\al\rg,\wt{e}(\mu))\cC_0(\lg\be\rg,\wt{e}(\nu))\ta^{-1}\bigr\}\\
&\qquad+e(\be\nu)
\bigl\{\cS_0(\al,\nu;\mu,-\be;\wh{q})-\cS_0(-\al,-\nu;-\mu,\be;\wh{q})\bigr\},
\end{align*}
in which the second term on the left side and the first term on the right side cancel out each other, and then $(-2\pi i)^{-1}$ is multiplied by both sides, to conclude (2.32) for $k=0$, after some rearrangements.
\end{proof}

The proof of Theorem~4 is thus complete. We lastly remark that Corollary~4.5  can be derived from (2.32) by noting the facts
\begin{align*}
\ze(-h)=-B_{h+1}/(h+1)-\de_{h0}
\qquad(h=0,1,\ldots),
\tag{7.25}
\end{align*}
coming from the case $(\ga,\ka)=(0,0)$ of (4.5), and  
\begin{align*}
B_0=1,
\qquad 
B_1=-1/2,
\qquad
B_2=1/6
\qquad
\text{and} 
\qquad
B_{2k+1}=0
\qquad
(k=1,2,\ldots)
\tag{7.26}
\end{align*}
(cf.~\cite[p.38, 1.13 (17); (19)]{erdelyi1953a}).  
\section{Proofs of Corollaries 4.6--4.9}
\begin{proof}[Proof of Corollary~4.6]
We first treat the case of non-positive integer weights. It follows from the case $k=1+2h$ 
$(h\geq1)$ of (7.17) with $-1/z$ instead of $z$, in view of (the second equalities in) (2.18) 
and (2.19), and (7.26) that   
\begin{align*}
&F_{\bbZ^2}'\Bigl(-2h;0,0;0,0;-\fr{1}{z}\Bigr)
\tag{8.2}\\
&\quad=\fr{(2h)!}{(2\pi\ta)^{2h}}\ze(1+2h)+2\pi(2h)!
\sum_{j=0}^{1+h}\fr{(-1)^hB_{2j}B_{2h+2-2j}}{(2j)!(2h+2-2j}\ta^{1-2j}\\
&\qquad+\fr{2(2h)!}{(2\pi\ta)^{2h}}\cS_{1+2h}(0,0;0,0;q)
\end{align*}
for $h\geq1$, while for $h=0$, from $\cB_1'(0,0,0)=-\pi i/2$ by (2.18), $\{\cB_2(2,0,0)\ta^{-s}\}'|_{s=0}=2\ze'(0)+\log\ta$ by (2.19), $S'_J(0;0,0;0,0;z)=\pi/6\ta+\pi i/2-\pi\ta/6$ for $J\geq1$ by (7.18), and from (7.19) that  
\begin{align*}
F'_{\bbZ^2}\Bigl(0;0,0;0,0;-\fr{1}{z}\Bigr)
&=2\ze'(0)-\log\ta+\fr{\pi\ta}{6}-\fr{\pi}{6\ta}+2\cS_1(0,0;0,0;q).
\tag{8.2}
\end{align*}
The combination of (8.1) with (8.2) therefore concludes, upon (3.1)--(3.3), the assertion 
(3.4) for $E_{-2h}(z)$ $(h=0,1,\ldots)$. 

We next treat the case of positive integer weights. It follows from the case $k=1-2h$ 
$(h\geq2)$ of (7.6) with $-1/z$ instead of $z$ that   
\begin{align*}
F_{\bbZ^2}\Bigl(2h;0,0;0,0;-\fr{1}{z}\Bigr)
&=\fr{(2\pi\ta)^{2h}}{(2h-1)!}\ze(1-2h)
+\fr{2(2\pi\ta)^{2h}}{(2h-1)!}\cS_{1-2h}(0,0;0,0;q)
\tag{8.3}
\end{align*}
for $h\geq2$, while for $h=1$, from (7.23) that
\begin{align*}
F_{\bbZ^2}\Bigl(2;0,0;0,0;-\fr{1}{z}\Bigr)
&=(2\pi\ta)^2\ze(-1)+2\pi\ta+2(2\pi\ta)^2\cS_{-1}(0,0;0,0;q).
\tag{8.4}
\end{align*}
The combination of (8.3) with (8.4) therefore concludes the assertion (3.4) for $E_{2h}(z)$ 
$(h=1,2,\ldots)$.
\end{proof}
\begin{proof}[Proof of Corollary~4.7]
We first see from (2.3), (2.13) and the partial fraction expansion for $1/\sin^2w$ that 
\begin{align*}
\cA(2,\al,0)=\ze(2,-\al)+\ze(2,\al)
=\sum_{-\al\neq l}\fr{1}{(\al+l)^2}=\fr{\pi^2}{\sin^2\pi\al}, 
\end{align*}
upon which (2.14) implies 
\begin{align*}
\Fz2{2}{\al,\be}{0,0}{z}
&=\fr{\de(\be)\pi^2}{\sin^2\pi\al}-4\pi^2\{\cS_{-1}(\be,0;0,\al;q)
+\cS_{-1}(-\be,0;0,-\al;q)\},
\end{align*}
and this, together with (3.1) for $k=2$, (3.10) and (7.25), concludes (3.11). 

We mention here in addition how (3.13) can be derived. First from (3.1) and (3.2) for $k=1$, (7.23)  
and 
\begin{gather*}
\cS_{-1}(0,0;0,\pm1/2;q)
=\sum_{l=1}^{\infty}\fr{(2l)p^{4l}}{1-p^{4l}}
-\sum_{l=1}^{\infty}\fr{(2l-1)p^{4l-2}}{1-p^{4l-1}},
\end{gather*}
we obtain the formula for $e_1(\bz)$. We have similarly
\begin{align*} 
e_2(\bz)
&=-\fr{\pi^2}{3}-8\pi^2\sum_{l=1}^{\infty}\fr{lp^l}{1+p^l},\\
e_3(\bz)
&=-\fr{\pi^2}{3}+8\pi^2
\biggl\{\sum_{l=1}^{\infty}\fr{(2l-1)p^{2l-1}}{1-p^{2l-1}}
-\sum_{l=1}^{\infty}\fr{(2l)p^{2l}}{1+p^{2l}}\biggr\},
\end{align*}
where the (first) $l$-sums on the right sides are both rewritten by (the logarithmic differentiation of) Euler's 
identity $(-q;q)_{\infty}=1/(q;q^2)_{\infty}$ (cf.~\cite[p.26, 2.1(2.43)]{venkatachaliengar2012}), 
yielding the formulae for $e_j(\bz)$ $(j=2,3)$ respectively. 
\end{proof}
\begin{proof}[Proof of Corollary~4.8]
Let $\vep>0$ be a sufficiently small real number. Suppose first that $\be\neq0$. Integrating 
the expression in (3.11) over the line segment from $\vep w$ to $w$, we find that 
\begin{align*}
&\fr{1}{w}-\int_{\vep w}^w\Bigl\{\wp(u\mid\bz)-\fr{1}{u^2}\Bigr\}du\\
&\quad=\fr{\pi^2}{3}E_2(z)(1-\vep)w
-2\pi i\bigl\{\cS_0(\be,0;0,\al;q)-\cS_0(-\be,0;0,-\al;q)\bigr\}\\
&\qquad+2\pi i\biggl\{\sum_{l=1}^{\infty}\fr{q^{(1+\vep w)l}}{1-q^l}
-\sum_{l=1}^{\infty}\fr{q^{(1-\vep w)l}}{1-q^l}
-\fr{q^{(\sgn\be)\vep w}}{1-q^{(\sgn\be)\vep w}}\biggr\}+\fr{1}{\vep w},
\end{align*}
whose limit case $\vep\to0^+$ concludes (3.15) for $\be\neq0$, since 
$q^{(\sgn\be)\vep w}/\{1-q^{(\sgn\be)\vep w}\}=(2\pi i\vep w)^{-1}+(\sgn\be)/2+O(\vep)$.  

Suppose next that $\be=0$, i.e. $w=\al$. Then we have similarly 
\begin{align*}
&\fr{1}{\al}-\int_{\vep \al}^{\al}\Bigl\{\wp(u\mid\bz)-\fr{1}{u^2}\Bigr\}du\\
&\quad=\fr{\pi^2}{3}E_2(z)(1-\vep)\al+\pi\cot\pi\al
-2\pi i\bigl\{\cS_0(0,0;0,\al;q)-\cS_0(0,0;0,-\al;q)\bigr\}\\
&\qquad+2\pi i\biggl\{\sum_{l=1}^{\infty}\fr{e(\vep\al l)q^l}{1-q^l}
-\sum_{l=1}^{\infty}\fr{e(-\vep\al l)q^l}{1-q^l}\biggr\}-\pi\cot\pi\vep\al+\fr{1}{\vep \al},
\end{align*}
whose limit case $\vep\to0^+$ concludes (3.15) for $\be=0$, since $\cot\pi\vep\al=(\pi\vep\al)^{-1}+O(\vep)$.  
\end{proof} 
\begin{proof}[Proof of Corollary~4.9]
Suppose first that $\be\neq0$. Integrating the expression in (3.15), similarly to the preceding proof, we obtain    
\begin{align*}
&\log w+\int_{\vep w}^w\Bigl\{\ze(u\mid\bz)-\fr{1}{u}\Bigr\}du\\
&\quad=\fr{\pi^2}{6}E_2(z)(1-\vep^2)w^2-(\sgn\be)\pi i(1-\vep)w
-\cS_1(\be,0;0,\al;q)-\cS_1(-\be,0;0,-\al;q)\\
&\qquad+\sum_{l=1}^{\infty}\fr{e(\vep wl)q^l}{l(1-q^l)}
+\sum_{l=1}^{\infty}\fr{e(-\vep wl)q^l}{l(1-q^l)}
-\log\bigl[1-e\{(\sgn\be)\vep w\}\bigr]+\log(\vep w),
\end{align*}
whose limit case $\vep\to0^+$ concludes (3.18) for $\be\neq0$, since 
$\log[1-e\{(\sgn\be)\vep w\}]=(\sgn\be)\pi i\vep w-(\sgn\be)\pi i/2+\log(2\sin\pi\vep w)$ 
and $\log(2\sin\pi\vep w)=\log(2\pi\vep w)+O(\vep)$.  

Suppose next that $\be=0$, i.e. $w=\al$. Then we have similarly 
\begin{align*}
&\log \al+\int_{\vep\al}^{\al}\Bigl\{\ze(u\mid\bz)-\fr{1}{u}\Bigr\}du\\
&\quad=\fr{\pi^2}{6}E_2(z)(1-\vep^2)\al^2+\log(2\sin\pi\al)
-\cS_1(0,0;0,\al;q)-\cS_1(0,0;0,-\al;q)\\
&\qquad+\sum_{l=1}^{\infty}\fr{e(\vep\al l)q^l}{l(1-q^l)}
+\sum_{l=1}^{\infty}\fr{e(-\vep\al l)q^l}{l(1-q^l)}-\log(2\sin\pi\vep\al)+\log(\vep\al),
\end{align*}
whose limit case $\vep\to0^+$ concludes (3.18) for $\be=0$. 

Lastly the remaining (3.19) is derived by noting that $\log(Z;q)_{\infty}
=\sum_{m=1}^{\infty}Z^m/m(1-q^m)$ for $|Z|<1$.
\end{proof}

\end{document}